\documentclass[12pt, a4paper]{amsart}
\usepackage{amssymb,latexsym}
\usepackage[top=35mm, bottom=25mm, left=30mm, right=30mm]{geometry}

\begin{document}
\title[Letterplace and co-letterplace ideals of posets]{Letterplace and 
co-letterplace ideals of posets}

\author{Gunnar Fl{\o}ystad}
\address{Matematisk institutt\\
        University of Bergen} 
\email{gunnar@mi.uib.no}
\author{Bj\o rn M\o ller Greve}
\address{Matematisk institutt\\
         University of Bergen}
\email{Bjorn.Greve@uib.no}

\author{J\"urgen Herzog}
\address{Matematisches Institut \\
         Universit\"at Essen}
\email{juergen.herzog@gmail.com}



\date{\today}

\maketitle


\theoremstyle{plain}
\newtheorem{theorem}{Theorem}[section]
\newtheorem{corollary}[theorem]{Corollary}
\newtheorem*{main}{Main Theorem}
\newtheorem{lemma}[theorem]{Lemma}
\newtheorem{proposition}[theorem]{Proposition}

\theoremstyle{definition}
\newtheorem{definition}[theorem]{Definition}

\theoremstyle{remark}
\newtheorem{notation}[theorem]{Notation}
\newtheorem{remark}[theorem]{Remark}
\newtheorem{example}[theorem]{Example}
\newtheorem{claim}{Claim}
\newtheorem{question}[theorem]{Question}


\newcommand{\psp}[1]{{{\bf P}^{#1}}}
\newcommand{\psr}[1]{{\bf P}(#1)}
\newcommand{\opw}{\op_{\psr{W}}}

\newcommand{\ini}[1]{\text{in}(#1)}
\newcommand{\gin}[1]{\text{gin}(#1)}
\newcommand{\kr}{{\Bbbk}}
\newcommand{\pd}{\partial}
\newcommand{\vardel}{\partial}
\renewcommand{\tt}{{\bf t}}


\newcommand{\coh}{{{\text{{\rm coh}}}}}
\newcommand{\Poset}{{\mathbf {Poset}}}


\newcommand{\modv}[1]{{#1}\text{-{mod}}}
\newcommand{\modstab}[1]{{#1}-\underline{\text{mod}}}

\newcommand{\sut}{{}^{\tau}}
\newcommand{\sumit}{{}^{-\tau}}
\newcommand{\til}{\thicksim}

\newcommand{\totp}{\text{Tot}^{\prod}}
\newcommand{\dsum}{\bigoplus}
\newcommand{\dprod}{\prod}
\newcommand{\lsum}{\oplus}
\newcommand{\lprod}{\Pi}

\newcommand{\La}{{\Lambda}}

\newcommand{\sirstj}{\circledast}

\newcommand{\she}{\EuScript{S}\text{h}}
\newcommand{\cm}{\EuScript{CM}}
\newcommand{\cmd}{\EuScript{CM}^\dagger}
\newcommand{\cmri}{\EuScript{CM}^\circ}
\newcommand{\cler}{\EuScript{CL}}
\newcommand{\clerd}{\EuScript{CL}^\dagger}
\newcommand{\clerri}{\EuScript{CL}^\circ}
\newcommand{\gor}{\EuScript{G}}
\newcommand{\gF}{\mathcal{F}}
\newcommand{\gG}{\mathcal{G}}
\newcommand{\gM}{\mathcal{M}}
\newcommand{\gE}{\mathcal{E}}
\newcommand{\gD}{\mathcal{D}}
\newcommand{\gI}{\mathcal{I}}
\newcommand{\gP}{\mathcal{P}}
\newcommand{\gK}{\mathcal{K}}
\newcommand{\gL}{\mathcal{L}}
\newcommand{\gS}{\mathcal{S}}
\newcommand{\gC}{\mathcal{C}}
\newcommand{\gO}{\mathcal{O}}
\newcommand{\gJ}{\mathcal{J}}
\newcommand{\gU}{\mathcal{U}}
\newcommand{\mm}{\mathfrak{m}}

\newcommand{\dlim} {\varinjlim}
\newcommand{\ilim} {\varprojlim}

\newcommand{\CM}{\text{CM}}
\newcommand{\Mon}{\text{Mon}}


\newcommand{\Kom}{\text{Kom}}


\newcommand{\EH}{{\mathbf H}}
\newcommand{\res}{\text{res}}
\newcommand{\inhom}{{\underline{\text{Hom}}}}
\newcommand{\Ext}{\text{Ext}}
\newcommand{\Tor}{\text{Tor}}
\newcommand{\ghom}{\mathcal{H}om}
\newcommand{\gext}{\mathcal{E}xt}
\newcommand{\id}{\text{{id}}}
\newcommand{\im}{\text{im}\,}
\newcommand{\codim} {\text{codim}\,}
\newcommand{\resol}{\text{resol}\,}
\newcommand{\rank}{\text{rank}\,}
\newcommand{\lpd}{\text{lpd}\,}
\newcommand{\coker}{\text{coker}\,}
\newcommand{\supp}{\text{supp}\,}
\newcommand{\Ad}{A_\cdot}
\newcommand{\Bd}{B_\cdot}
\newcommand{\Fd}{F_\cdot}
\newcommand{\Gd}{G_\cdot}


\newcommand{\sus}{\subseteq}
\newcommand{\sups}{\supseteq}
\newcommand{\pil}{\rightarrow}
\newcommand{\vpil}{\leftarrow}
\newcommand{\rpil}{\leftarrow}
\newcommand{\lpil}{\longrightarrow}
\newcommand{\inpil}{\hookrightarrow}
\newcommand{\pils}{\twoheadrightarrow}
\newcommand{\projpil}{\dashrightarrow}
\newcommand{\dotpil}{\dashrightarrow}
\newcommand{\adj}[2]{\overset{#1}{\underset{#2}{\rightleftarrows}}}
\newcommand{\mto}[1]{\stackrel{#1}\longrightarrow}
\newcommand{\vmto}[1]{\stackrel{#1}\longleftarrow}
\newcommand{\mtoelm}[1]{\stackrel{#1}\mapsto}

\newcommand{\eqv}{\Leftrightarrow}
\newcommand{\impl}{\Rightarrow}

\newcommand{\iso}{\cong}
\newcommand{\te}{\otimes}
\newcommand{\into}[1]{\hookrightarrow{#1}}
\newcommand{\ekv}{\Leftrightarrow}
\newcommand{\equi}{\simeq}
\newcommand{\isopil}{\overset{\cong}{\lpil}}
\newcommand{\equipil}{\overset{\equi}{\lpil}}
\newcommand{\ispil}{\isopil}
\newcommand{\vvi}{\langle}
\newcommand{\hvi}{\rangle}
\newcommand{\susneq}{\subsetneq}
\newcommand{\sgn}{\text{sign}}


\newcommand{\xd}{\check{x}}
\newcommand{\ortog}{\bot}
\newcommand{\tL}{\tilde{L}}
\newcommand{\tM}{\tilde{M}}
\newcommand{\tH}{\tilde{H}}
\newcommand{\tvH}{\widetilde{H}}
\newcommand{\tvh}{\widetilde{h}}
\newcommand{\tV}{\tilde{V}}
\newcommand{\tS}{\tilde{S}}
\newcommand{\tT}{\tilde{T}}
\newcommand{\tR}{\tilde{R}}
\newcommand{\tf}{\tilde{f}}
\newcommand{\ts}{\tilde{s}}
\newcommand{\tp}{\tilde{p}}
\newcommand{\tr}{\tilde{r}}
\newcommand{\tfst}{\tilde{f}_*}
\newcommand{\empt}{\emptyset}
\newcommand{\bfa}{{\bf a}}
\newcommand{\bfb}{{\bf b}}
\newcommand{\bfd}{{\bf d}}
\newcommand{\bfe}{{\mathbf e}}
\newcommand{\bff}{{\mathbf f}}
\newcommand{\bfl}{{\bf \ell}}
\newcommand{\la}{\lambda}
\newcommand{\bfen}{{\mathbf 1}}
\newcommand{\ep}{\epsilon}
\newcommand{\en}{r}
\newcommand{\tu}{s}
\newcommand{\Ga}{\Gamma}
\newcommand{\hatphi}{\hat{\phi}}

\newcommand{\ome}{\omega_E}

\newcommand{\bevis}{{\bf Proof. }}
\newcommand{\demofin}{\qed \vskip 3.5mm}
\newcommand{\nyp}[1]{\noindent {\bf (#1)}}
\newcommand{\demo}{{\it Proof. }}
\newcommand{\demodone}{\demofin}
\newcommand{\parg}{{\vskip 2mm \addtocounter{theorem}{1}  
                   \noindent {\bf \thetheorem .} \hskip 1.5mm }}

\newcommand{\red}{{\text{red}}}
\newcommand{\lcm}{{\text{lcm}}}


\newcommand{\dl}{\Delta}
\newcommand{\cdel}{{C\Delta}}
\newcommand{\cdelp}{{C\Delta^{\prime}}}
\newcommand{\dlst}{\Delta^*}
\newcommand{\Sdl}{{\mathcal S}_{\dl}}
\newcommand{\lk}{\text{lk}}
\newcommand{\lkd}{\lk_\Delta}
\newcommand{\lkp}[2]{\lk_{#1} {#2}}
\newcommand{\del}{\Delta}
\newcommand{\delr}{\Delta_{-R}}
\newcommand{\dd}{{\dim \del}}
\newcommand{\dis}[1]{\underline{#1}}

\renewcommand{\aa}{{\bf a}}
\newcommand{\bb}{{\bf b}}
\newcommand{\cc}{{\bf c}}
\newcommand{\xx}{{\bf x}}
\newcommand{\yy}{{\bf y}}
\newcommand{\zz}{{\bf z}}
\newcommand{\mv}{{\xx^{\aa_v}}}
\newcommand{\mF}{{\xx^{\aa_F}}}

\newcommand{\Symm}{\text{Sym}}
\newcommand{\pnm}{{\bf P}^{n-1}}
\newcommand{\opnm}{{\go_{\pnm}}}
\newcommand{\ompnm}{\omega_{\pnm}}

\newcommand{\pn}{{\bf P}^n}
\newcommand{\hele}{{\mathbb Z}}
\newcommand{\nat}{{\mathbb N}}
\newcommand{\rasj}{{\mathbb Q}}

\newcommand{\dt}{\bullet}
\newcommand{\st}{\hskip 0.5mm {}^{\rule{0.4pt}{1.5mm}}}              
\newcommand{\disk}{\scriptscriptstyle{\bullet}}

\newcommand{\cF}{F_\dt}
\newcommand{\pol}{f}

\newcommand{\Rn}{{\mathbb R}^n}
\newcommand{\An}{{\mathbb A}^n}
\newcommand{\frg}{\mathfrak{g}}
\newcommand{\PW}{{\mathbb P}(W)}

\newcommand{\pos}{{\mathcal Pos}}
\newcommand{\g}{{\gamma}}
\newcommand{\ov}[1]{\overline{#1}}
\newcommand{\ovv}[1]{\overline{\overline{#1}}}
\newcommand{\cJ}{\mathcal J}
\newcommand{\cC}{\mathcal C}
\newcommand{\cK}{\mathcal K}
\newcommand{\cS}{\mathcal S}
\newcommand{\ord}{\prec}
\newcommand{\ordl}{\preceq}
\newcommand{\ords}{\prec_s}
\newcommand{\ordls}{\preceq_s}

\newcommand{\ovm}{\overline{m}}
\newcommand{\ovmp}{\overline{m^\prime}}
\newcommand{\op}{{op}}
\newcommand{\bihom}[2]{\overset{#1}{\underset{#2}{\rightleftarrows}}}

\def\opn#1#2{\def#1{\operatorname{#2}}} 
\opn\Hom{Hom}
\opn\height{height}
\opn\projdim{proj\,dim}
\opn\Min{Min}

\def\CC{{\mathbb C}}
\def\GG{{\mathbb G}}
\def\ZZ{{\mathbb Z}}
\def\NN{{\mathbb N}}
\def\RR{{\mathbb R}}
\def\OO{{\mathbb O}}
\def\QQ{{\mathbb Q}}
\def\VV{{\mathbb V}}
\def\PP{{\mathbb P}}
\def\EE{{\mathbb E}}
\def\FF{{\mathbb F}}
\def\AA{{\mathbb A}}

\def\pp{{\mathfrak p}}

\begin{abstract}
To a natural number $n$, a finite partially ordered set $P$
and a poset ideal $\cJ$ in the poset $\Hom(P,[n])$ of isotonian maps
from $P$ to the chain on $n$ elements,  we associate 
two monomial ideals, the letterplace ideal $L(n,P;\cJ)$ and the
co-letterplace ideal $L(P,n;\cJ)$.
These ideals give a unified understanding
of a number of ideals studied in monomial ideal theory in recent years.
By cutting down these ideals by regular 
sequences of variable differences we obtain:
multichain ideals and generalized Hibi type
ideals, initial ideals of determinantal ideals, strongly stable ideals,
$d$-partite $d$-uniform ideals, Ferrers ideals, edge ideals of cointerval
$d$-hypergraphs, and uniform face ideals.
\end{abstract}

\section*{Introduction}

Monomial ideals arise as initial ideals of polynomial ideals.
For natural classes of polynomial ideals, like determinantal ideals
of generic matrices, of generic symmetric matrices, and of 
Pfaffians of skew-symmetric matrices, 
their Gr\"obner bases and initial ideals \cite{Stu}, \cite{HeTr},
\cite{CHT} have been computed. 
In full generality one has the class of Borel-fixed
ideals (in characteristic $0$ these are the strongly stable ideals), 
which arise as initial ideals
of any polynomial ideal after a generic change of coordinates. 
This is a very significant class in computational algebra \cite{BaSt},
\cite{Green}.

Monomial ideals have also grown into an active research area in itself.
In particular one is interested in their resolution, and 
in properties like shellability
(implying Cohen-Macaulayness), and linear quotients (implying linear 
resolution). Classes that in particular have been studied are 
generalized Hibi ideals \cite{EHM}, 
$d$-uniform 
hypergraph ideals \cite{NaRe}, Ferrers ideals \cite{CoNa1}, \cite{NaRe},
uniform face ideals \cite{Co} and \cite{HeHi}, and cointerval 
$d$-hypergraph ideals \cite{DoEn}. 

This article presents a unifying framwork for these seemingly varied
classes of monomial 
ideals by introducing
the classes of letterplace and co-letterplace ideals associated to a
a natural number $n$, a
finite partially ordered set $P$, and a poset ideal $\cJ$ in the poset 
$\Hom(P,[n])$ of isotone maps $P \pil [n]$, where $[n]$ is the chain on 
$n$ elements. Many basic results of the abovementioned articles follow
from the general results we give here.

As it turns out, most of the abovementioned classes of ideals are not 
letterplace and co-letterplace ideals. Rather they {\it derive} from
these ideals by dividing out by a {\it regular sequence} of variable
differences. The main technical results of the present paper is to give criteria
for when such a sequence of variable differences is regular,
Theorems  \ref{QLPThmNPR}, \ref{QLPThmPNR}, \ref{CLPThmPNRJ}, and
\ref{CLPThmNPRJ}.
Thus we get classes of ideals, with
the same homological properties as letterplace and co-letterplace ideals.
Also we show that letterplace and co-letterplace ideals in themselves
may not be derived from other monomial ideals by cutting down by 
a regular sequence of varibale differences, Lemma \ref{Lem-LPSep}. 
These ideals therefore
have the flavor of beeing "free" objects
in the class of monomial ideals. We see this as accounting for many of
their nice properties, see \cite{DaFlNeCoLP} and \cite{FlNe}.

In \cite{EHM} V.Ene, F.Mohammadi, and the third author introduced
the classes of generalized Hibi ideals, associated to natural numbers
$ s \leq n$ and a finite partially ordered set $P$, and investigated
these ideals. This article both generalizes this, and
provides a hightened understanding of these ideals. There is
an extra parameter $s$ involved here, but we show that these ideals
can be understood as the letterplace ideal associated to the natural
number $n$ and 
the partially ordered set $P \times [n-s+1]$, after we divide
this ideal out by a regular sequence of variable differences.

\medskip
The {\it $n$'th letterplace ideal} associated to $n$ and $P$ is the monomial ideal,
written $L(n,P)$, 
generated by all monomials
\[ x_{1,p_1}x_{2,p_2} \cdots x_{n,p_n} \]
where $p_1 \leq p_2 \leq \cdots \leq p_n$ is a multichain in $P$.
These ideals $L(n,P)$ were shown in \cite{EHM} to be Cohen-Macaulay (CM) ideals
(actually shellable) of codimension equal to the cardinality of $P$.
These ideals are all generated in degree $n$.
If $P$ is the antichain on $d$ elements, then $L(n,P)$ is a complete
intersection, and defines a quotient ring of mulitplicity $n^d$.
If $P$ is the chain on $d$ elements, then $L(n,P)$ is the initial ideal
of the ideal of maximal minors of a generic $n \times (n+d-1)$ matrix.
The quotient ring has multiplicity $\binom{n+d-1}{d}$. These
are respectively the maximal and minimal multiplicity of CM ideals
of codimension $d$ generated in degree $n$. As $P$ varies 
among posets of cardinality $d$ we therefore
get ideals interpolating between these extremal cases. 

Its Alexander dual, the {\it $n$'th co-letterplace ideal} associated to $n$
and $P$, written $L(P,n)$, is the ideal generated by all monomials 
\[ \prod_{p \in P} x_{p,i_p} \]
where $p < q$ in $P$ implies $i_p \leq i_q$. This is an ideal with linear
quotients \cite{EHM}, and therefore linear resolution.

\medskip
The Cohen-Macaulay ideals $L(n,P)$ are all generated in a single
degree $n$. To obtain CM ideals with varying degrees on generators,
we now add an extra layer of structure.
Given a poset ideal $\cJ$ in the poset $\Hom(P,[n])$ of isotone
maps $P \pil [n]$, we get first more generally the co-letterplace subideal
$L(P,n;\cJ) \sus L(P,n)$. This ideal also has linear quotients, 
Theorem \ref{CLPThmPosid}.
When $P$ is a chain on $d$ elements, all strongly stable
ideals generated in degree $d$ are regular quotients of these co-letterplace
ideals, Example \ref{Subsec-CLPSS}. 
As $P$ varies we therefore get a substantial generalization of
the class of strongly stable ideals generated in a single degree, 
and with much the same nice
homological behaviour.

The Alexander dual of $L(P,n;\cJ)$ we denote by $L(n,P;\cJ)$, and
is explicitly described in Theorem \ref{CLPThmAd}. Since the
former ideal has linear resolutions, by
\cite{ER} the latter ideal is  
Cohen-Macaulay and it contains the letterplace ideal $L(n,P)$. 
Even when $P$ is the chain on $d$ elements, all $h$-vectors of embedding
dimension $d$ of 
graded Cohen-Macaulay ideals (in a polynomial ring), may be realized for
such ideals $L(P,n;\gJ)$. 
This is therefore a very large class of Cohen-Macaulay ideals.

\medskip
  Dividing out a monomial ring $S/I$ by an difference of variables $x_a - x_b$,
corresponds to setting the variables $x_a$ and $x_b$ equal in $I$ to obtain
a new monomial ideal $J$. In this article we therefore naturally 
introduce the notion of separation
of variables, Definition \ref{Def-QLPSep}, 
of a monomial ideals: $I$ is obtained from 
$J$ by a {\it separation} 
of variables, if $x_a - x_b$ is a regular element for $S/I$. Surprisingly
this simple and natural notion does not seem to have been a topic of 
study in itself before for monomial ideals, but see \cite{Fl}.
In particular the behaviour of Alexander duality when
dividing out by such a regular sequence of variable differences, is given in
Proposition \ref{QLPProAlexDual}, and Theorems \ref{QLPThmAd}
and \ref{QLPThmAdJ}.

\medskip
  After the appearance of this article as a preprint, a number of further
investigations has been done on
letterplace and co-letterplace ideals. The article \cite{KJM} studies
more generally ideals $L(P,Q)$ where both $P$ and $Q$ are 
finite partially orded sets, and \cite{HeQuSh} investigates Alexander duality
of such ideals. In \cite{DaFlNeLP} resolutions of letterplace ideals
$L(n,P)$ are studied, in particular their multigraded Betti numbers are
computed. \cite{DaFlNeCoLP} gives explicitly the
linear resolutions of co-letterplace ideal $L(P,n;\gJ)$, thereby
generalizing the Eliahou-Kervaire resolution for strongly stable ideals 
generated in a single degree. It computes
the canonical modules of the Stanley-Reisner rings of 
letterplace ideals $L(n,P;\gJ)$. They
have the surprising property of being multigraded ideals in 
these Stanley-Reisner rings. A related and 
remarkable consequence is that the simplicial complexes
associated to letterplace ideals $L(n,P;\gJ)$ are triangulations of balls.
Their boundaries are therefore triangulations of spheres, and this
class of sphere triangulations 
comprehensively generalizes the class of Bier spheres \cite{BPSZ}.
The notion of separation is further investigated in \cite{HeRa} and in 
\cite{AlBiHeLu}, which shows that separation corresponds
to a deformation of the monomial ideal, and identifies the deformation
directions in the cotangent cohomology it corresponds to. 
In \cite{FlNe} deformations of letterplace ideals $L(2,P)$ are 
computed when the Hasse diagram has the structure of a rooted tree. 
The situation is remarkably nice. These ideals are unobstruced, 
and the full deformation family can be explicitly computed. This 
deformed family has a polynomial ring as a base ring, and the ideal
of the full family is a rigid ideal. In some simple example cases these
are the ideals of $2$-minors of a generic $2 \times n$ matrix, and the ideal
of Pfaffians of a generic skew-symmetric $5 \times 5$ matrix.

\medskip
The organization of the paper is as follows.
In Section \ref{LPSec} we define ideals $L(Q,P)$ associated to pairs
of posets $Q$ and $P$. In particular for the totally ordered poset $[n]$ on
$n$ elements, we introduce the letterplace ideals
$L([n],P)$ and co-letterplace ideals $L(P,[n])$. We investigate
how they behave under Alexander duality.
In Section \ref{QLPSec} we 
study when a sequence of variable differences
is regular for letterplace and co-letterplace ideals. We also 
define the notion of separation.
Section \ref{ExQLPSec} gives classes of ideals, including generalized
Hibi ideals
and initial ideals
of determinantal ideals, which are quotients of letterplace ideals
by a regular sequence.
Section \ref{FacetSec} describes in more detail the generators and
facets of various letterplace and co-letterplace ideals.
Section \ref{CLPSec} considers poset ideals $\cJ$ in $\Hom(P,[n])$ and 
the associated co-letterplace ideal $L(P,n;\cJ)$. We show it has
linear resolution, and compute its Alexander dual $L(n,P;\gJ)$.
Section \ref{ExCLPSec} gives classes of ideals which are quotients
of co-letterplace ideals by a regular sequence. This includes
strongly stable ideals, $d$-uniform $d$-partite hypergraph ideals,
Ferrers ideals, and uniform face ideals.
The last sections \ref{AdSec} and \ref{RegSec}
contain proofs of basic results in this paper on when sequences of
variable differences are regular, and how Alexander duality behaves
when cutting down by such a regular sequence.

\section{Letterplace ideals and their Alexander duals} \label{LPSec}

If $P$ is a partially ordered set (poset), a {\it poset ideal} $J \sus P$ is
a subset of $P$ such that $q \in J$ and $p \in P$ with $p \leq q$,
implies $p \in J$. The term {\it order ideal} is also much used in the
literature for this notion. If $S$ is a subset of $P$, the
poset ideal generated by $S$ is the set of all elements $p \in P$
such that $p \leq s$ for some $s \in S$.

\subsection{Isotone maps}
Let $P$ and $Q$ be two partially ordered sets. A map $\phi: Q \pil P$
is {\it isotone} or {\it order preserving}, if $q \leq q^\prime$
implies $\phi(q) \leq \phi(q^\prime)$. The set of isotone maps is denoted
$\Hom(Q,P)$. It is actually again a partially ordered set
with $\phi \leq \psi$ if $\phi(q) \leq \psi(q)$ for all $q \in Q$.
The following will be useful.

\begin{lemma} \label{LPLemFix}
If $P$ is a finite partially ordered set with a unique maximal 
or minimal element, then
an isotone map $\phi : P \pil P$ has a fix point.
\end{lemma}

\begin{proof}
We show this in case $P$ has a unique minimal  element $p = p_0$.
Then $p_1 = \phi(p_0)$ is $\geq p_0$. If $p_1 > p_0$, let 
$p_2 = \phi(p_1) \geq \phi(p_0) = p_1$. If $p_2 > p_1$ we continue.
Since $P$ is finite, at some stage $p_n = p_{n-1}$ and since
$p_n = \phi(p_{n-1})$, the element $p_{n-1}$ is a fix point.
\end{proof}

\subsection{Alexander duality}
Let $\kr$ be a field. If $R$ is a set, denote by $\kr[x_R]$ the polynomial
ring in the variables $x_r$ where $r$ ranges over $R$. 
If $A$ is a subset of $R$ denote by $m_A$ the monomial $\Pi_{a \in A} x_a$.

Let $I$ be a squarefree monomial ideal in a polynomial ring $\kr[x_R]$,
i.e. its generators are monomials of the type $m_A$.
It corresponds to a simplicial complex $\Delta$ on the vertex set $R$,
consisting of all $S \sus R$, called faces of $\Delta$, 
such that $m_S \not \in I$.

The Alexander dual $J$ of $I$, written $J = I^A$,
 may be defined in different ways.
Three definitions are the following.

\noindent 1. The Alexander dual $J$ is the monomial ideal in $\kr[x_R]$
whose monomials are
those with nontrivial common divisor with 
every monomial in $I$. 

\noindent 2. The Alexander dual $J$ is the ideal generated by all
monomials $m_S$ where the $S$ are complements in $R$ of faces of 
$\Delta$.

\noindent 3. If  $I = \cap_{i = 1}^r \pp_r$ is a decomposition into
prime monomial ideals $\pp_i$ where $\pp_i$ is generated by the
variables $x_a$ as a ranges over the subset $A_i$ of $R$, then
$J$ is the ideal generated by the monomials $m_{A_i}, \, i = 1, \ldots, r$.
(If the decomposition is a minimal primary decomposition, the
$m_{A_i}$ is a minimal generating set of $J$.)

\subsection{Ideals from Hom-posets}
To an isotone map $\phi: Q \pil P$ we associate its graph
$\Gamma \phi \sus Q \times P$ where $\Gamma \phi = 
\{ (q, \phi(q)) \, | \, q \in Q \}$. 
As $\phi$ ranges over $\Hom(Q,P)$, the monomials $m_{\Gamma \phi}$
generate a monomial ideal in $\kr[x_{Q \times P}]$
which we denote by $L(Q,P)$. 
More generally, if $\cS$ is a subset of $\Hom(Q,P)$ we get ideals
$L(\cS)$ generated by $m_{\Gamma \phi}$ where $\phi \in \cS$. 

If $R$ is a subset of the product $Q \times P$, we denote by
$R^\tau$ the subset of $P \times Q$ we get by switching coordinates.
As $L(Q,P)$ is an ideal in $\kr[x_{Q \times P}]$, we may also consider
it as an ideal in $\kr[x_{P \times Q}]$. In cases where we need to be
precise about this, we write it then as $L(Q,P)^\tau$.

\medskip
If $Q$ is the totally ordered poset on $n$
elements $Q = [n] = \{1 < 2 < \cdots < n \}$, we call 
$L([n],P)$, written simply $L(n,P)$, the $n$'th {\it letterplace ideal}
of $P$. It is generated by the monomials
\[ x_{1,p_1} x_{2,p_2} \cdots x_{n,p_n} \mbox{ with } p_1 \leq p_2 \leq
\cdots \leq p_n. \]
This is precisely the same ideal as the multichain ideal
$I_{n,n}(P)$ defined in \cite{EHM} (but with indices switched). 
The ideal $L(P,[n])$, written simply $L(P,n)$, is the $n$'th {\it co-letterplace
ideal} of $P$. In \cite{EHM} it is denoted $H_n(P)$ and is called
a generalized Hibi type ideal. For some background on the
letterplace notion, see Remark \ref{LPRemLett} at the end of this section.

The following is Theorem 1.1(a) in \cite{EHM}, suitably reformulated.
Since it is a very basic fact, we include a proof of it. 

\begin{proposition} \label{LPProAdual}
The ideals $L(n,P)$ and $L(P,n)^\tau$ are Alexander dual in 
$\kr[x_{[n] \times P}]$.
\end{proposition}



\begin{proof}
Let $L(n,P)^A$ be the Alexander dual of $L(n,P)$. 
First we show $L(P,n) \sus L(n,P)^A$.
This is equivalent to: For any $\phi \in \Hom([n],P)$ and any
$\psi \in \Hom(P,[n])$, the graphs $\Ga \phi$ and $\Ga \psi^\tau$ intersect
in $[n] \times P$. 
Let $i$ be a fix point for $\psi \circ \phi$. 
Then $i \mto{\phi} p \mto{\psi}
i$ and so $(i,p)$ is in both $\Ga \phi$ and $\Ga \psi^\tau$. 

Secondly, given a squarefree monomial $m$ in $L(n,P)^A$
we show that it is divisible by a monomial in $L(P,n)$. 
This will show that $L(n,P)^A \sus L(P,n)$ and force
equality here.
So let the monomial $m$ correspond to the subset $F$ of $P \times [n]$. 
It intersects all graphs $\Ga \phi^\tau$ where  $\phi \in \Hom([n],P)$. 
We must show it contains a graph $\Ga \psi$ where $\psi \in \Hom(P, [n])$. 
Given $F$, let $\cJ_n = P$ and let $\cJ_{n-1}$ be the poset ideal of $P$
generated by all $p \in \cJ_n = P$ such that $(p,n) \not \in F$. 
Inductively let $\cJ_{i-1}$ be the poset ideal in $\cJ_i$ generated by
all $p$ in $\cJ_i$ with $(p,i)$ not in $F$. 

\begin{claim} $\cJ_0 = \emptyset.$
\end{claim}

\begin{proof} Otherwise let $p \in \cJ_0$. Then there is $p \leq p_1$
with $p_1 \in \cJ_1$ and $(p_1,1) \not \in F$. Since $p_1 \in \cJ_1$ there is
$p_1 \leq p_2$ with $p_2 \in \cJ_2$ such that $(p_2,2) \not \in F$. 
We may continue this and get a chain $p_1 \leq p_2 \leq \cdots \leq p_n$
with $(p_i,i)$ not in $F$. 
But this contradicts $F$ intersecting all graphs $\Ga \phi$ where
$\phi \in \Hom([n],P)$.
\end{proof}

We thus get a filtration of poset ideals
\[  \emptyset = \cJ_0 \sus \cJ_1 \sus \cdots \sus \cJ_{n-1} \sus 
\cJ_n = P. \]
This filtration corresponds to an isotone map $\psi : P \pil [n]$. 

\begin{claim} $\Ga \psi$ is a subset of $F$.
\end{claim}

\noindent {\it Proof.} Let $(p,i) \in \Ga \psi$. Then  $p \in \cJ_i \backslash
\cJ_{i-1}$ and so $p \not \in \cJ_{i-1}$. Thus $(p,i) \in F$. 

\end{proof}

\begin{remark}
The case $n =2$ was shown in \cite{HeHi2} where
the ideal $H_P$ generated by $\prod_{p \in J} x_p \prod_{q \in P \backslash J}
y_q$ as $J$ varies over the poset ideals in $P$, 
was shown to be Alexander dual to the ideal 
generated by $x_p y_q$ where $p \leq q$. In \cite{HeHi2} it is
also shown that the $L(2,P)$ are precisely the edge ideals of 
bipartite Cohen-Macaulay graphs.
\end{remark}

\begin{remark}
That $L(m,n)$ and $L(n,m)$ are Alexander dual is Proposition 4.5
of \cite{FlVa}. There the elements of these ideals 
are interpreted as paths in a $m \times n$ matrix with generic linear forms
$(x_{ij})$ and the generators of the ideals are the products of 
the variables in these paths.
\end{remark}

\subsection{Alexander dual of $L(Q,P)$}
In general $L(Q,P)$ and $L(P,Q)$ are not Alexander dual.
This is easily checked if for instance $Q$ and $P$ are antichains
of sizes $\geq 2$. However we have the following.

\begin{proposition}
Suppose $Q$ has a unique maximal or minimal element.
The least degree of a generator of the Alexander dual $L(Q,P)^A$ 
and of $L(P,Q)$ are both $d = |P|$ and the degree $d$ parts of these 
ideals are equal. In particular, since $L(P,Q)$ is generated in this 
degree $d$, it is contained in $L(Q,P)^A$. 
\end{proposition}

Note that the above is equivalent to say that the minimal primes
of $L(Q,P)$ of height $\leq |P|$ are precisely the
\[ \pp_\psi=(\{x_{\psi(p),p}\:\; p\in P\}), \quad \mbox{where } 
\psi \in \Hom(P,Q) .\] 

\begin{proof}
We show that:
\begin{itemize}
\item[1.] $L(Q,P) \subset \pp_\psi$ for all $\psi\in \Hom(P,Q)$.
\item[2.] $\pp_\psi$ is a minimal prime of $L(Q,P)$.
\item[3.] Any minimal prime $\pp$ of $L(Q,P)$ is $\pp = \pp_\psi$
for some $\psi$. 
\end{itemize}
This will prove the proposition.

\noindent 1. Given $\phi\in \Hom(Q,P)$ and $\psi\in \Hom(P,Q)$. 
We have to show that 
$m_\phi=\prod_{q\in Q}x_{q,\phi(q)}\in \pp_\psi$. 
By Lemma \ref{LPLemFix} $\psi \circ \phi$ has a fix point $q$,
and let 
$p=\phi(q)$. Then $\psi(p)=q$. Therefore,  $x_{q,p}$ is a factor of 
$m_\phi$ and a generator of $\pp_\psi$. This implies that $m_\phi\in\pp_\psi$.

\noindent 2. Next we show that $\pp_\psi$ is a minimal prime ideal of $L(Q,P)$. 
Suppose this is not the case.  Then we may skip one of its generators, 
say $x_{\psi(p), p}$,  to obtain the prime ideal $\pp\subset \pp_\psi$ with 
$L(Q,P)\subset \pp$. Let $\phi\in \Hom(Q,P)$ be the constant isotone map with 
$\phi(q)=p$ for all $q\in Q$. Then $m_{\phi}=\prod_{q\in  Q}x_{q,p}\in L(Q,P)$.
  Since no factor of $m_\phi$ is divisible by a generator of $\pp$, 
it follows that $L(Q,P)$ is not contained in $\pp$, a contradiction.

\noindent 3. Now let $\pp$ be any minimal prime ideal of $L(Q,P)$.  
Since $L(Q,P)\subset \pp$ it follows as in the previous paragraph that for
 each $p\in P$ there exists an element $\psi(p)\in Q$ such that 
$x_{\psi(p),p}\in \pp$. This show that $\height L(Q,P)=|P|$. Assume now that 
$\height \pp=|P|$. Then $\pp=(\{x_{\psi(p),p}\:\; p\in P\})$. It remains to be
 shown that $\psi\: P\to Q$ is isotone. Suppose this is not the case. Then 
there exist $p, p'\in P$ such that $p<p'$ and $\psi(p)\not\leq \psi(p')$.
 Let $\phi : Q\to P$ the map with $\phi(q)=p$ if $q\leq \psi(p')$ and 
$\phi(q)=p'$ if $q\not\leq \psi(p')$. Then $\phi$ is isotone, and it 
follow that $m_\phi= \prod_{q\leq \psi(p')}x_{q,p}\prod_{q\not\leq \psi(p')}x_{q,p'}$ 
does not belong to $\pp$, a contradiction.
\end{proof}

\begin{remark} In \cite{HeQuSh} they determine precisely when
$L(P,Q)$ and $L(Q,P)$ are Alexander dual, for 
finite posets $P$ and $Q$.
\end{remark}

\begin{remark} \label{LPRemLett}
Let $X = \{ x_1, \ldots, x_n \}$ be an alphabet. The letterplace
correspondence is a way to encode non-commutative monomials
$x_{i_1}x_{i_2}\cdots x_{i_r}$ in $\kr\langle X \rangle$ by commutative 
polynomials $x_{i_1,1} \cdots x_{i_r,r}$ in $\kr[X \times \NN]$. 
It is due to G.-C. Rota who again attributed it to a physicist, apparently
Feynman. D. Buchsbaum has a survey article \cite{Bu} on letterplace
algebra, and the use of these technique in the resolution of Weyl modules.
Recently \cite{ScaLev} use letterplace ideals in computations of 
non-commutative Gr\"obner bases.
\end{remark}

\section{Quotients of letterplace ideals} \label{QLPSec}

A chain $c$ in the product of two posets $Q \times P$ is said to be 
{\it left strict} if for two elements in the chain,
$(q,p) < (q^\prime, p^\prime)$ implies $q < q^\prime$.
Analogously we define right strict. The chain is {\it bistrict}
if it is both left and right strict.

An isotone map of posets $\phi: Q \times P \pil R$ is said to have  
{\it  left strict chain fibers} if 
all its fibers $\phi^{-1}(r)$ are left strict chains in $Q \times P^\op$.
Here $P^\op$ is the opposite poset of $P$, i.e. $p \leq^\op p^\prime$ 
in $P^{op}$ iff $p^\prime \leq p$ in $P$.  

The map $\phi$ gives a map of linear spaces $\phi_1 : \langle x_{Q \times P} 
\rangle \pil \langle x_R \rangle$ (the brackets here mean the
$\kr$-vector space spanned by the set of variables). 
The map $\phi_1$ induces a map of 
polynomial rings $\hat{\phi} : \kr[x_{Q \times P}] \pil \kr[x_R]$. 
In the following $B$ denotes  a basis for the kernel of the map
of degree one forms $\phi_1$, 
consisting of differences $x_{q,p} - x_{q^\prime,p^\prime}$ with 
$\phi_1(q,p) = \phi_1(q^\prime,p^\prime)$. 

\begin{theorem} \label{QLPThmNPR} Given an isotone map
$\phi: [n] \times P \pil R$ which has left strict chain fibers.
Then the basis $B$ is a regular sequence of the ring
$\kr[x_{[n] \times P}]/L(n,P)$. 
\end{theorem} 



\begin{theorem} \label{QLPThmPNR} Given an isotone map
$\psi: P \times [n] \pil R$ which has left strict chain fibers.
Then the basis $B$ is a regular sequence of the ring
$\kr[x_{P \times [n]}]/L(P,n)$. 
\end{theorem} 

We shall prove these in Section \ref{RegSec}.
For now we note that they require distinct proofs, with
the proof of Theorem \ref{QLPThmPNR} the most delicate.



In the setting of Theorem \ref{QLPThmNPR}, 
we let $L^\phi(n,P)$ be the ideal generated by 
the image of the $n$'th letterplace ideal $L(n,P)$ in $\kr[x_R]$,
and in the setting of Theorem \ref{QLPThmPNR},
we let $L^\psi(P,n)$ be the ideal generated by 
the image of the $n$'th co-letterplace ideal $L(P,n)$ in $\kr[x_R]$,
Note that $L^\phi(n,P)$ is a squarefree ideal iff in the above
the fibers $\phi^{-1}(r)$ are bistrict chains in $[n] \times P^{op}$, 
and similarly
$L^\psi(P,n)$ is a squarefree ideal iff the fibers 
$\psi^{-1}(r)$ are bistrict chains in $P \times [n]^{op}$.

We get the following consequence of the above Theorems \ref{QLPThmNPR}
and \ref{QLPThmPNR}.

\begin{corollary}
The quotient rings $\kr[x_{[n] \times P}]/L(n,P)$ and $\kr[x_R]/L^\phi(n,P)$
have the same graded Betti numbers. Similarly for 
$\kr[x_{P \times [n]}]/L(P,n)$ and $\kr[x_R]/L^\psi(P,n)$.
\end{corollary}

\begin{proof} We prove the first statement. 
Let $L^{\im \phi}(n,P)$ be the image of 
$L(n,P)$ in $\kr[x_{\im \phi}]$, and $S = R\backslash \im \phi$. 
Thus $\kr[x_{\im \phi}]/L^{\im \phi}(n,P)$ is a quotient of 
$\kr[x_{[n] \times P}]/L(n,P)$ by a regular sequence, and 
$\kr[x_R] / L^{\phi}(n,P)$ is 
$\kr[x_{\im \phi}]/L^{\im \phi}(n,P)  \te_\kr \kr[x_S]$.
\end{proof}



For the poset $P$ consider the multichain ideal $I(n,P)$ 
in $\kr[x_P]$ generated by monomials 
$x_{p_1} x_{p_2} \cdots x_{p_n}$ where $p_1 \leq p_2 \leq \cdots \leq p_n$
is a multichain of length $n$ in $P$. The quotient 
$\kr[x_P]/I(n,P)$ is clearly artinian since $x_p^n$ is in 
$I(n,P)$ for every $p \in P$. 

\begin{corollary} The ring $\kr[x_P]/I(n,P)$ is an artinian 
reduction of $\kr[x_{[n] \times P}]/L(n,P)$ by a regular sequence.
In particular $L(n,P)$ is a Cohen-Macaulay ideal.
It is Gorenstein iff $P$ is an antichain.
\end{corollary}

\begin{proof} The first part is because the 
map $[n] \times P \pil P$ fulfills the criteria
of Theorem \ref{QLPThmNPR} above. 
An artinian ideal is Gorenstein iff it is a complete intersection.
Since all $x_p^n$ are in $I(n,P)$, this holds iff there are no more
generators of $I(n,P)$, which means precisely that
$P$ is an antichain.
\end{proof}

This recovers part of Theorem 2.4 of \cite{EHM} showing that $L(n,P)$
is Cohen-Macaulay. The Gorenstein case above is Corollary 2.5 of 
\cite{EHM}.

\begin{remark}
The multigraded Betti numbers of the resolution of $L(n,P)$ is
described in \cite{DaFlNeLP}, as well as other properties
of this resolution.
\end{remark}

Recall that a squarefree monomial ideal is bi-Cohen-Macaulay, \cite{FlVa}, iff
both the ideal and its Alexander dual are Cohen-Macaulay ideals.

\begin{corollary} $L(n,P)$ is bi-Cohen-Macaulay iff $P$ is totally ordered.
\end{corollary}

\begin{proof} Since $L(n,P)$ is Cohen-Macaulay, it is 
bi-Cohen-Macaulay iff it has a  linear resolution,
\cite{EaRe}. Equivalently $I(n,P)$ in $\kr[x_P]$ has a linear resolution.
But since $I(n,P)$ gives an artinian quotient ring, this is equivalent
to $I(n,P)$ being the $n$'th power of the maximal ideal.
In this case every monomial $x_p^{n-1} x_q$ is in $I(n,P)$
and so every pair $p,q$ in $P$ is comparable. Thus $P$ is totally ordered.
Conversely, if $P$ is totally ordered, then clearly $I(n,P)$ is
the $n$'th power of the maximal ideal.
\end{proof}

\medskip
 
\begin{definition} \label{Def-QLPSep}
Let $R^\prime \pil R$
be a surjective map of sets with $R^\prime$ of cardinality one more than $R$, 
and let $r_1 \neq r_2$ in $R^\prime$ map to the same element in $R$. 
Let $I$ be a monomial ideal in $k[x_R]$. A monomial ideal $J$ in 
$\kr[x_{R^\prime}]$ is a {\it separation} of $I$ if 
i) $I$ is the image of $J$ by the natural map $\kr[x_{R^\prime}] \pil \kr[x_R]$,
ii) $x_{r_1}$ occurs in
some minimal generator of $J$ and similarly for $x_{r_2}$,
and iii) $x_{r_1} - x_{r_2}$ is a regular element of $\kr[x_{R^\prime}]/J$.

The ideal $I$ is {\it separable} if it has some separation $J$.
Otherwise it is {\it inseparable}. If $J$ is obtained from $I$ by
a succession of separations, we also call $J$ a {\it separation}
of $I$. We say that $I$ is a {\it regular quotient by
variable differences} of $J$, or simply a {\it regular quotient} of 
$J$.  If $J$ is {\it inseparable}, then
$J$ is a {\it separated model} for $I$. 
\end{definition}

This notion also occurs in \cite{Fl} where inseparable monomial ideals
are called maximal. The canonical example of a separation of
a non-squarefree monomial ideal is of course its polarization.

\begin{lemma} 
\label{Lem-LPSep}
Let $I$ be an ideal generated by a subset of the
generators of $L(Q,P)$. Then $I$ is inseparable.
\end{lemma}



\begin{proof} Let $R^\prime \pil Q \times P$ be a surjective map 
with $R^\prime$  of cardinality one more than $Q \times P$. 
Suppose there is a monomial ideal $J$ in $\kr[x_{R^\prime}]$ 
which is a separation 
of $I$. Let $a$ and $b$ in  $R^\prime$ both map to $(q,p)$. 
For any other element of $R^\prime$, we identify it with its image 
in $Q \times P$.
Suppose $m = x_am_0$
in $J$ maps to a generator $x_{q,p} m_0$ of $L(Q,P)$, and
$m^\prime  = x_bm_1$ maps to another generator of $L(Q,P)$. Then 
$m_0$ does not contain a variable $x_{q,p^\prime}$ with first index $q$, 
and similarly for $m_1$. Note that the least common multiple $m_{01}$
of $m_0$ and $m_1$ does not contain a variable with first index $q$.
Hence $m_{01}$ is not in $L(Q,P)$ and so $m_{01}$ is not in $J$. But 
$(x_b - x_a)m_{01}$ is in $J$ since $x_bm_{01}$ and $x_am_{01}$ are in $J$. 
By the regularity of $x_b - x_a$ this implies $m_{01}$ in $J$, a contradiction.
\end{proof} 


As we shall see, many naturally occurring monomial ideals are separable
and have separated models which are letterplace ideals $L(n,P)$
or are generated by a subset of the generators of co-letterplace ideals
$L(P,n)$. 

\begin{remark} \label{QLPRem} 
In \cite[Section 2]{Fl} the first author shows that the 
separated models of the squarefree power $(x_1, \ldots, x_n)^{n-1}_{sq}$
are in bijection with trees on $n$ vertices.
\end{remark}

\medskip
We now consider the Alexander dual of $L^\phi(n,P)$.

\begin{theorem} \label{QLPThmAd}
Let $\phi : [n] \times P \pil R$ be an isotone map such that
the fibers $\phi^{-1}(r)$ are bistrict chains in $[n] \times P^\op$.
Then the ideals $L^\phi(n,P)$ and $L^{\phi^\tau}(P,n)$ are
Alexander dual.
\end{theorem}

We prove this in Section \ref{AdSec}.

\begin{remark} 
The Alexander dual of the squarefree power in Remark \ref{QLPRem} is the
squarefree power $(x_1, \ldots, x_n)^2_{sq}$. Separations of this ideal
are studied by H.Lohne, \cite{Loh}. In particular he describes how the
separated models are also in bijection with trees on $n$ vertices.
\end{remark}


\section{Examples of regular quotients of letterplace ideals}
\label{ExQLPSec}

The ideals which originally inspired this
paper are the multichain ideals of \cite{EHM}. 

\subsection{Multichain ideals}
\label{MuChSubsec}

Let $P_m$ be $P \times [m]$ where $m \geq 1$. Consider the surjective
map
\begin{align*}
[s] \times P_m & \pil   P \times [m+s-1] \\
(i,p,a) & \mapsto  (p,a+i-1).
\end{align*}

This map has left strict chain fibers.
The image of $L(s,P_m)$ in $\kr[x_{P \times [m+s-1]}]$ is exactly
the multichain ideal $I_{m+s-1,s}(P)$ of \cite{EHM}. 
This is the ideal generated by monomials
\[ x_{p_1,i_1} x_{p_2, i_2} \cdots x_{p_s, i_s} \]
where 
\[ p_1 \leq \cdots \leq p_s, \quad 1 \leq i_1 < \cdots < i_s \leq m+s-1. \] 
By Theorem \ref{QLPThmNPR} it is obtained from the $s$'th letterplace 
ideal $L(s,P_m) = L(s, P \times [m])$ by
cutting down by a regular sequence.
Thus we recover the fact, \cite[Thm. 2.4]{EHM}, 
that these ideals are Cohen-Macaulay. 
\medskip

The Alexander dual of $L(s,P_m)$ is $L(P_m,s)$. 
An element $r$ of $\Hom(P \times [m], [s])$ may be represented by 
sequences
\[ 1 \leq r_{p1} \leq \cdots \leq r_{pm} \leq s \]
such that for each $p \leq q$ we have $r_{pj} \leq r_{qj}$. 

The element $r$ gives the monomial generator in $L(P_m,s)$
\[ m_r = \underset{p \in P}{\prod}  \prod_{i = 1}^m x_{p,i,r_{pi}}.\]

By Theorem \ref{QLPThmAd}, the Alexander dual of the multichain 
ideal $I_{m+s-1,s}(P)$ is then generated by 
\[ \underset{p \in P}{\prod}  \prod_{i = 1}^m x_{p,t_{pi}}, 
\quad 1 \leq t_{p1} < t_{p2} < \cdots < t_{p_m} \leq m+s-1 \]
(where $t_{pi} = r_{pi} + i-1$)
such that $p < q$ implies $t_{pj} \leq t_{qj}$.
These are exactly the generators of the squarefree power ideal
$L(P, s+m-1)^{\langle m \rangle}$. This recovers Theorem 1.1(b)
in \cite{EHM}. 

\subsection{Initial ideals of determinantal ideals: two-minors}

We now let $P = [n]$ and $s = 2$. Let $e,f \geq 0$. 
There are isotone maps
\begin{align*} [2] \times [n] \times [m] = [2] \times P_m 
& \mto{\phi_{e,f}} [n+e] \times [m+f] \\
(1,a,b) & \mapsto (a,b) \\
(2,a,b) & \mapsto (a+e,b+f)
\end{align*}

These maps have left strict chain fibers and we
get the ideal $L^{\phi_{e,f}}(2,P_m)$. 

\begin{itemize}
\item When $(e,f) = (0,1)$ we are in the situation of the previous
Subsection \ref{MuChSubsec}, and we get
the multichain ideal $I_{m+1,2}([n])$. 
\item When $(e,f) = (1,0)$ we get the multichain  ideal $I_{n+1,2}([m])$.
\item When $(e,f) = (1,1)$ we get the ideal in $\kr[x_{[n+1] \times [m+1]}]$
generated by monomials $x_{i,j}x_{i^\prime,j^\prime}$ where $i < i^\prime$
and $j < j^\prime$. This is precisely the initial  ideal $I$ of 
the ideal of two-minors of a generic $(n+1) \times (m+1)$ matrix
of linear forms $(x_{i,j})$ with respect to a suitable monomial 
order with respect to a diagonal term order, \cite{Stu}.
\end{itemize}

In particular all of $I_{m+1,2}([n]), I_{n+1,2}([m])$ and $I$ have the
same graded Betti numbers and the same $h$-vector, 
the same as $L(2,[n] \times [m])$.

Particularly noteworthy is the following: The ideal of two-minors of the generic
$(n+1) \times (m+1)$ matrix is the homogeneous ideal of the 
Segre product of $\PP^m \times \PP^n$ in $\PP^{nm+n+m}$. 
By J.Kleppe \cite{Kl}, any deformation  of a generic determinantal ideal
is still a generic determinantal ideal. So if this Segre embedding is 
obtained from
a variety $X$ in a higher dimensional projective space,
by cutting it down by a regular sequence of linear forms, this $X$ must
be a cone over the Segre embedding.
Thus we cannot ``lift'' the ideal of two minors to an ideal in 
a polynomial ring with more variables than $(n+1)(m+1)$. 
However its initial ideal
may be separated to the monomial ideal
$L(2, [n] \times [m])$ with $2nm$ variables.

\medskip
Varying $e$ and $f$, we get a whole family of ideals 
$L^{\phi_{e,f}}(2,[n] \times [m])$
with the same Betti numbers as the initial  ideal of the ideal of 
two-minors. When $e = 0 = f$ we get an artinian reduction, not of the
initial  ideal of the ideal of two-minors, but of its separated model 
$L(2,[n] \times [m])$. When $e \geq n+1$ and $f \geq m+1$, the 
map $\phi_{e,f}$ is injective and $L^{\phi_{e,f}}(2,[n] \times [m])$
is isomorphic to the ideal generated by  $L(2,[n] \times [m])$ 
in a polynomial ring with more variables.

\subsection{Initial ideals of determinantal ideals: higher minors}

We may generalize to arbitrary $s$ and two weakly increasing sequences
\[ \bfe = (e_1 = 0, e_2, \ldots, e_s), \quad \bff = (f_1 = 0, f_2, \ldots, 
f_s). \]

We get isotone maps 
\begin{align*}
[s] \times [n] \times [m] & \mto{}  [n+e_s] \times [m+f_s] \\
(i,a,b) & \mapsto  (a+e_i, b+f_i) 
\end{align*}

\begin{itemize}
\item When $\bfe  = (0, \ldots, 0)$ and $\bff = (0,1, \ldots, s-1)$
we get the multichain ideal $I_{m+s-1,s}([n])$. 
\item When  $\bfe = (0,1, \ldots, s-1)$ and $\bff = (0, \ldots, 0)$
we get the multichain ideal $I_{n+s-1,s}([m])$. 
\item When $\bfe = (0,1, \ldots, s-1)$ and $\bff = (0,1, \ldots, s-1)$
we get the ideal $I$ generated by monomials
\[ x_{i_1, j_1} x_{i_2, j_2} \cdots x_{i_{s},j_s} \]
where $i_1 < \cdots < i_s$ and $j_1 < \cdots < j_s$. 
This is the initial ideal $I$ of the ideal of $s$-minors of a general 
$(n+s-1) \times (m+s-1)$ matrix $(x_{i,j})$ with respect to a diagonal
term order, \cite{Stu}.
\end{itemize}

We thus see that this initial ideal $I$ has a lifting to 
$L(s,[n] \times [m])$ with $snm$ variables, in contrast to the
$(n+s-1)(m+s-1)$ variables which are involved in the ideal
of $s$-minors. We get maximal minors when, say $m = 1$. Then the
initial ideal $I$ involves $sn$ variables. So in this case 
the initial ideal $I$ involves the same number of variables
as $L(s,[n])$, i.e. the generators of these two ideals are in one to one
correspondence by a bijection of variables. 

\subsection{Initial ideal of the ideal of two-minors of a symmetric
matrix}
Let $P = \Hom([2],[n])$. The elements here may be identified
with pairs $(i_1,i_2)$ where $1 \leq i_1 \leq i_2 \leq n$. 
There is an isotone map
\begin{align*} 
\phi : [2] \times \Hom([2],[n]) & \pil \Hom([2],[n+1]) \\
(1,i_1,i_1) & \mapsto (i_1,i_2) \\
(2,i_1,i_2) & \mapsto (i_1+1, i_2 +1).
\end{align*}
This map has left strict chain fibers, and 
we get a regular quotient ideal $L^{\phi}(2, \Hom([2],[n]))$,
generated by $x_{i_1,i_2} x_{j_1,j_2}$ where $i_1 < j_1$ and $i_2 < j_2$
(and $i_1 \leq i_2$ and $j_1 \leq j_2$).
This is the initial ideal of the ideal generated by $2$-minors
of a symmetric matrix of size $n+1$, see \cite[Sec.5]{CHT}.

\subsection{Ladder determinantal ideals}
Given a poset ideal $\cJ$
in $[m] \times [n]$. This gives the letterplace ideal $L(2,\cJ)$. 
There is a map 
\begin{align*}
\phi : [2] \times \cJ & \pil [m+1] \times [n+1] \\
(1,a,b) & \mapsto (a,b) \\
(2,a,b) & \mapsto (a+1,b+1)
\end{align*}
The poset ideal $\cJ$ is sometimes also called a one-sided ladder in 
$[m] \times [n]$.
The ideal $L^\phi(2, \cJ)$ is the initial ideal of the ladder determinantal ideal
associated to $\cJ$, \cite[Cor.3.4]{Nar}.
Hence we recover the fact that these are Cohen-Macaulay,
\cite[Thm.4.9]{HeTr}.

\subsection{Pfaffians}
Let $T(n)$ be the poset in $[n] \times [n]$ consists
of all $(a,b)$ with $a+b \leq n+1$.
Then $L^{\phi}(2,T(n))$ is the initial ideal of the ideal of 
$4$-Pfaffians of a skew-symmetric matrix of rank $n+3$,
\cite[Sec.5]{HeTr}. It is also the initial ideal of the Grassmannian
$G(2,n+3)$, \cite[Ch.6]{HeEn}.

The poset $T(2)$ is the $V$ poset. 
The letterplace ideals $L(n,T(2))$ are the initial ideals of 
the $2n$-Pfaffians of a generic $(2n+1) \times (2n+1)$ skew-symmetric matrix,
by \cite[Thm.5.1]{HeTr}. The variables $X_{i,2n+2-i}$ in loc.cit. correspond
to our variables $x_{i,(1,1)}$ for $i = 1, \ldots, n$, 
the variables $X_{i+1,2n+2-i}$ correspond to the $x_{i,(2,1)}$ and the 
$X_{i,2n+1-i}$ to the $x_{i,(1,2)}$.

\section{Description of facets and ideals} \label{FacetSec}

As we have seen $\Hom(Q,P)$ is itself a poset. The product $P \times Q$
makes the category of posets $\Poset$ into a symmetric monoidal category,
and with this internal $\Hom$, it is a symmetric monoidal closed category
\cite[VII.7]{MacL},
i.e. there is an adjunction of functors
\[ \Poset \bihom{- \times P}{\Hom(P,-)} \Poset \]
so that 
\[  \Hom(Q \times P, R) \iso \Hom(Q, \Hom(P,R)). \]
This is an isomorphism of posets.
Note that the distributive lattice $D(P)$ associated to $P$, 
consisting of the poset ideals in $P$, identifies with $\Hom(P,[2])$. 
In particular $[n+1]$ identifies as $\Hom([n],[2])$.
The adjunction above gives isomorphisms between the following posets.

\begin{itemize}
\item[1.] $\Hom([m], \Hom(P,[n+1]))$
\item[2.] $\Hom([m] \times P,  [n+1]) = \Hom([m] \times P,  \Hom([n],[2]))$
\item[3.] $\Hom([m] \times P \times [n], [2])$
\item[4.] $\Hom([n] \times P, \Hom([m],[2])) = \Hom([n] \times P, [m+1])$
\item[5.] $\Hom([n], \Hom(P,[m+1]))$
\end{itemize}


These $\Hom$-posets normally give distinct letterplace or
co-letterplace ideals associated to the same underlying (abstract) poset. 
There are natural bijections between the generators. The degrees of the 
generators are normally distinct, and so they
have different resolutions.

Letting $P$ be the one element poset, we get from 2.,3., and 4. above
isomorphisms
\begin{eqnarray} \label{FacetLigPart} \quad
\Hom([m],[n+1]) \iso \Hom([m] \times [n], [2]) \iso 
\Hom([n],[m+1]). 
\end{eqnarray}

An element $\phi$ in $\Hom([m],[n+1])$ identifies as a partition
$\la_1 \geq \cdots \geq \la_m \geq 0$ with $m$ parts of sizes $\leq n$, by
$\phi(i) = \la_{m+1-i} + 1$. The left and right side of the isomorphisms
above give the correspondence between a partition and its dual.
This poset is the Young lattice. In Stanley's book \cite{Sta},
Chapter 6 is about this lattice, there denoted $L(m,n)$.

\medskip
Letting $m = 1$ we get by 2.,3., and 5. isomorphims:
\[ \Hom(P,[n+1]) \iso \Hom(P \times [n], [2]) \iso \Hom(n,D(P))\]
and so we have ideals
\[ L(P,n+1), \quad L(P \times [n], [2]), \quad L(n,D(P))\]
whose generators are naturally in bijection with each other, in particular
with elements of $\Hom([n],D(P))$, which are chains of poset ideals in 
$D(P)$:
\begin{equation} \label{QLPLigFilt}
 \emptyset = \cJ_0 \sus \cJ_1 \sus \cdots \sus \cJ_{n} \sus 
\cJ_{n+1} = P.
\end{equation}
The facets of the simplicial complexes associated to their Alexander duals 
\[ L(n+1, P), \quad L(2, P \times [n]), \quad L(D(P),n),  \]
are then in bijection with elements of $\Hom([n],D(P))$.

For a subset $A$ of a set $R$, let $A^c$ denote
its complement $R \backslash A$. 

\medskip
1. The facets of the simplicial complex associated to 
$L(n+1,P)$ identifies as the complements $(\Gamma \phi)^c$
of graphs of $\phi : P \pil [n+1]$.
This is because these facets correspond to the complements of the
set of variables in the generators in the Alexander dual 
$L(P,n+1)$ of $L(n+1,P)$. 

\medskip
For isotone maps $\alpha : [n+1] \times P \pil R$
having bistrict chain fibers, the associated simplicial complex of the ideal 
$L^\alpha(n+1,P)$, has also facets
in one-to-one correspondence with $\phi : P \pil [n+1]$, 
or equivalently $\phi^\prime : [n] \pil D(P)$, but the
precise description varies according to $\alpha$.

\medskip
2. The facets of the simplicial complex associated to $L(2, P \times
[n])$ identifies as the complements $(\Gamma \phi)^c$ 
of the graphs of $\phi: P \times [n] \pil [2]$. Alternatively the
facets identifies as the graphs $\Gamma \phi^\prime$ of 
$\phi^\prime : P \times [n] \pil [2]^{op}$.

\medskip
3. Let 
\[ \alpha : [2] \times P \times [n] \pil P \times [n+1], \quad
(a,p,i) \mapsto (p,a+i-1). \]
The ideal $L^\alpha(2, P \times [n])$ is the multichain ideal $I_{n+1,2}(P)$.
The generators of this ideal are $x_{p,i}x_{q,j}$ where $p \leq q$ and 
$i < j$. The facets of the simplicial complex associated to this
ideal are the graphs $\Gamma \phi$ of $\phi : P \pil [n+1]^{op}$.

\section{Co-letterplace ideals of poset ideals} \label{CLPSec}

\subsection{The ideal $L(P,n;\cJ)$}
Since $\Hom(P,[n])$ is itself a partially ordered set, we can consider
poset ideals $\cJ \sus \Hom(P,[n])$ and form the subideal $L(P,n;\cJ)$
of $L(P,n)$ generated by the monomials $m_{\Gamma \phi}$ where $\phi \in \cJ$. 
We call it the {\it co-letterplace ideal of the poset ideal $\cJ$}.
For short we often write $L(\gJ)$ and call it simply a co-letterplace ideal.
For the notion of linear quotients we refer to \cite{HeHiMon}.

\begin{theorem}
\label{CLPThmPosid}
Let $\cJ$ be a poset ideal in $\Hom(P, [n])$. Then $L(P,n;\cJ )$
has linear quotients, and so it has linear resolution.
\end{theorem}

\begin{proof}
We extend the partial order $\leq$ on $\cJ$ to a total order, 
denoted $\leq^t$, and define an order on the generators of 
$L(\gJ)$ be setting $m_{\Gamma \psi}\geq  m_{\Gamma \phi}$ if and only if 
$\psi\leq^t \phi$. We claim that $L(\cJ)$ has linear quotients with respect 
to this total order of the monomial generators of $L(\cJ)$. Indeed,
let $m_{\Gamma \psi}>m_{\Gamma \phi}$ where $\psi \in \cJ$. 
Then $\psi <^t \phi$, and hence there 
exists $p\in P$ such that
$\psi(p)<\phi(p)$. We choose  a $p\in P$ which is minimal with this property. 
Therefore, if $q<p$, then $\phi(q)\leq\psi(q)\leq \psi(p)<\phi(p)$. We set
\[
\psi'(r)= \left\{ \begin{array}{ll}
       \psi(r), & \;\textnormal{if  $r=p$}, \\
       \phi(r), & \;\textnormal{otherwise.}
        \end{array} \right.
\]
Then $\psi'\in \Hom(P,n)$ and $\psi'<\phi$ for the original order on $P$. 
It follows that $\psi'\in \cJ$, 
and $m_{\Gamma \psi'}>m_{\Gamma \phi}$. Since $(m_{\Gamma \psi'}):m_{\Gamma \phi}=(x_{p,\psi(p)})$ and since 
$x_{p,\psi(p)}$ divides $m_{\Gamma \psi}$, the desired conclusion follows.
\end{proof}

\begin{remark} \label{CLPRemWOrd}
One may fix a maximal element $p \in P$. The statement above still
holds if $\cJ$ in $\Hom(P,[n])$ is a poset ideal for the weaker
partial order $\leq^w$ on $\Hom(P,[n])$ where $\phi \leq^w \psi$ if 
$\phi(q) \leq \psi(q)$ for $q \neq p$ and $\phi(p) = \psi(p)$.
Just let the total order still refine the standard partial order on the 
$\Hom(P,[n])$.
Then one deduces either $\psi^\prime \leq^w \phi$ or $\psi^\prime \leq^w \psi$.
In either case this gives $\psi^\prime \in \cJ$.
\end{remark}

%
%

\medskip
For an isotone map $\phi : P \pil [n]$, we define the set
\begin{equation} \label{CLPLigLam} 
\La\phi = \{ (p,i) \, | \, \phi(q) \leq i<\phi(p) 
\text{ for all } q<p\}. 
\end{equation}
It will in the next subsection play a role somewhat analogous
to the graph $\Gamma \phi$.
For $\phi \in \cJ$ we let $J_\phi$ be the ideal generated by all $m_{\Gamma \psi}$ with 
$m_{\Gamma \psi}>m_{\Gamma \phi}$, where we use the total order in the proof
of Theorem \ref{CLPThmPosid} above. 
In analogy to \cite[Lemma 3.1]{EHM} one obtains:

\begin{corollary}
\label{colon}
Let $\phi\in \cJ$. Then  $J_\phi:m_{\Gamma \phi}$ is 
$\{ x_{p,i} \, | \, (p,i) \in \La \phi \}$.
\end{corollary}

\begin{proof}
The inclusion $\subseteq$ has been shown in the proof of 
Theorem \ref{CLPThmPosid}. 
Conversely, 
let $x_{p,i}$ be an element of the right hand set.
We set
\[
\psi(r)= \left\{ \begin{array}{ll}
       i, & \;\textnormal{if  $r=p$}, \\
       \phi(r), & \;\textnormal{otherwise.}
        \end{array} \right.
\] 
Then $m_{\Gamma \psi}\in J_\phi$ and $(m_{\Gamma \psi}):m_{\Gamma \phi}=(x_{p,i})$. This proves the other 
inclusion.
\end{proof}

\begin{corollary}
The projective dimension of $L(P,n;\cJ)$ is the maximum of the
cardinalities $|\La \phi|$ for $\phi \in \cJ$.
\end{corollary}

\begin{proof} 
This follows by the above Corollary~\ref{colon}
and Lemma 1.5 of \cite{HeTa}.
\end{proof}

\begin{remark}
By  \cite[Cor.3.3]{EHM} the projective dimension of $L(P,n)$ is 
$(n-1)s$ where $s$ is the size of a maximal antichain in $P$.
It is not difficult to work this out as a consequence of the above
when $\cJ = \Hom(P,[n])$.

An explicit form of the minimal free resolution of $L(P,n)$
is given in \cite[Thm. 3.6]{EHM}, and this is generalized
to $L(P,n;\cJ)$ in \cite{DaFlNeCoLP}.
\end{remark}

\subsection{Regular quotients of $L(P,n;\cJ)$}
We now consider co-letterplace ideals of poset ideals when
we cut down by a regular sequence of variable differences.
The following generalizes Theorem \ref{QLPThmPNR} and we prove
it in Section \ref{RegSec}.

\begin{theorem} \label{CLPThmPNRJ} Given an isotone map
$\psi: P \times [n] \pil R$ with left strict chain fibers. 
Let $\cJ$ be a poset ideal in $\Hom(P,n)$.
Then the basis $B$ (as defined before Theorem \ref{QLPThmNPR}) 
is a regular sequence for the ring
$\kr[x_{P \times [n]}]/L(P,n;\cJ)$. 
\end{theorem} 

\subsection{Alexander dual of $L(P,n;\cJ)$}

 We describe the Alexander dual of $L(\cJ) = L(P,n;\cJ)$ when
$\cJ$ is a poset ideal in $\Hom(P,[n])$. We denote this Alexander
dual ideal as $L(\cJ)^A = L(n,P;\cJ)$. 
Note that since $L(P,n;\cJ) \sus L(P,n)$, 
the Alexander dual $L(n,P;\cJ)$ contains the letterplace ideal $L(n,P)$, and
since $L(P,n;\cJ)$ has linear resolution, the Alexander dual
$L(n,P;\cJ)$ is a Cohen-Macaulay ideal,
by \cite{EaRe}. Recall the set $\La \phi$ defined above (\ref{CLPLigLam}), 
associated to a map $\phi \in \Hom(P,[n])$. 

\begin{lemma} \label{AdJLemJJc}
Let $\cJ$ be a poset ideal in $\Hom(P,[n])$. Let $\phi \in \cJ$
and $\psi$ be in the complement $\cJ^c$. Then 
$\La\psi \cap \Gamma \phi$ is nonempty.
\end{lemma}

\begin{proof}
There is some $p \in P$ with $\psi(p) > \phi(p)$. Choose $p$ to be minimal 
with this property, and
let $i = \phi(p)$. 
If $(p,i)$ is not in $\La\psi$, there must be $q < p$ with 
$\psi(q) > i = \phi(p) \geq \phi(q)$. But this contradicts $p$ being
minimal. Hence $(p,i) = (p, \phi(p))$ is both in $\Gamma \phi$
and $\La\psi$. 
\end{proof}

\begin{lemma} \label{AdJLemSupB}
Let $S$ be a subset of $P \times [n]$ which is disjoint from  $\Gamma \phi$
for some $\phi$ in $\Hom(P,[n])$. If $\phi$ is a minimal such element
w.r.t. the
partial order on $\Hom(P,[n])$, then $S \supseteq \La \phi$. 
\end{lemma}

\begin{proof}
Suppose $(p,i) \in \La \phi$ and $(p,i)$ is not in $S$. 
Define $\phi^\prime : P \pil [n]$ by 
\[ \phi^\prime(q) = \begin{cases} \phi(q), & q \neq p \\
                                  i, & q = p
\end{cases} \]
By definition of $\La \phi$ we see that $\phi^\prime$ is an isotone map,
and $\phi^\prime < \phi$. But since $S$ is disjoint from 
$\Gamma \phi$, we see that it is also disjoint from  $\Gamma \phi^\prime$.
This contradicts $\phi$ being minimal. Hence every 
$(p,i) \in \La \phi$ is also in $S$. 
\end{proof}

For a subset $\cS$ of $\Hom(P,[n])$ define $K(\cS) \sus \kr[x_{[n] \times P}]$
to be the ideal generated by the monomials $m_{\La \phi^\tau}$ where
$\phi \in \cS$. 

\begin{theorem} \label{CLPThmAd}
The Alexander dual $L(n,P;\cJ)$ is $L(n,P) + K(\cJ^c)$. 
This is a Cohen-Macaulay ideal of codimension $|P|$.
\end{theorem}

\begin{proof} It is Cohen-Macaulay, in fact shellable, since
  $L(\cJ)$ has linear quotients by Theorem \ref{CLPThmPosid}.
The facets of the simplicial complex corresponding to $L(\cJ)^A$
are the complements of the generators of $L(\cJ)$. Hence
these facets have codimension $|P|$.

To prove the first statement we show the following.

\noindent 1. The right ideal is contained in the Alexander dual of the
left ideal: Every monomial in $L(n,P) + K(\cJ^c)$ has non-trivial 
common divisor with  every monomial in $L(\cJ)$.

\noindent 2. The Alexander dual of the left ideal is contained in
the right ideal: If $S \sus [n] \times P$ intersects every $\Gamma \phi^\tau$ 
where $\phi \in \cJ$, the monomial  $m_S$ is in $L(n,P) + K(\cJ^c)$.

\medskip
\noindent 1a. Let $\psi \in \Hom([n],P)$. 
Since $L(n,P)$ and $L(P,n)$ are Alexander dual, 
$\Gamma \psi \cap \Gamma \phi^\tau$
is non-empty for every $\phi \in \Hom(P,[n])$ and so in particular for
every $\phi \in \cJ$.

\noindent 1b. If $\psi \in \cJ^c$ then $\La \psi \cap \Gamma \phi$ is nonempty
for every $\phi \in \cJ$ by Lemma \ref{AdJLemJJc}.

Suppose now $S$ intersects every $\Gamma \phi^\tau$ where $\phi$ is in $\cJ$.

\noindent 2a. If $S$ intersects every $\Gamma \phi^\tau$ 
where $\phi$ is in $\Hom(P,[n])$,
then since $L(n,P)$ is the Alexander dual of $L(P,n)$, the monomial
$m_S$ is in $L(n,P)$. 

\noindent 2b. If $S$ does not intersect $\Gamma \phi^\tau$
where $\phi \in \cJ^c$, then by Lemma \ref{AdJLemSupB}, for a minimal
such $\phi$ we will have $S \supseteq \La \phi^\tau$.
Since $S$ intersects $\Gamma \phi^\tau$ for all $\phi \in \cJ$, 
a minimal such $\phi$ is in $\cJ^c$. 
Thus $m_S$ is divided
by $m_{\La \phi^\tau}$ in $K(\cJ^c)$.
\end{proof}

\begin{remark} For a more concrete example, to Stanley-Reisner
ideals with whiskers, see the
end of Subsection \ref{ExCLPSubsecFace}.
\end{remark}

\begin{remark} In \cite[Thm.5.1]{DaFlNeCoLP} it is shown
that the simplicial complex corresponding to $L(n,P;\cJ)$ is a 
triangulation of a ball. Its boundary is then a triangulation
of a sphere. This gives a comprehensive generalization  of 
Bier spheres, \cite{BPSZ}. In \cite[Sec.4]{DaFlNeCoLP} there is also
a precise description of the canonical module of the Stanley-Reisner
ring of $L(n,P;\cJ)$, as an ideal in this Stanley-Reisner ring.
\end{remark}

\subsection{Regular quotients of the Alexander dual $L(n,P;\cJ)$}
We now take the Alexander dual of $L(P,n;\cJ)$ and cut it
down by a regular sequence of variable differences.
We then get a generalization of Theorem \ref{QLPThmNPR} and we prove
it in Section \ref{RegSec}.

\begin{theorem} \label{CLPThmNPRJ} Given an isotone map
$\phi: [n] \times P \pil R$ with left strict chain fibers. 
Let $\cJ$ be a poset ideal in $\Hom(P,n)$.
Then the basis $B$ (as defined before Theorem \ref{QLPThmNPR}) 
is a regular sequence for the ring
$\kr[x_{[n] \times P}]/L(n,P;\cJ)$. 
\end{theorem}

\section{Examples of regular quotients of co-letterplace ideals}
\label{ExCLPSec}

We give several examples of quotients of co-letterplace ideals
which have been studied in the literature in recent years.

\subsection{Strongly stable ideals:}
 {\it Poset ideals in $\Hom([d],[n])$.} 

\label{Subsec-CLPSS}
Elements of $\Hom([d],[n])$ are in
one to one correspondence with monomials in $\kr[x_1, \ldots, x_n]$
of degree $d$: A map $\phi$ gives the monomial 
$ \Pi_{i = 1}^d x_{\phi(i)}$. By this association, 
the poset ideals in $\Hom([d], [n])$ 
are in one to one correspondence with strongly stable ideals
in $\kr[x_1, \ldots, x_n]$ generated in degree $d$. 

Consider the projections $[d] \times [n] \mto{p_2} [n]$. 
The following is a consequence of Theorem \ref{CLPThmPosid}
and Theorem \ref{CLPThmPNRJ}.
\begin{corollary} \label{OrdCorSs}
Let $\cJ$ be a poset ideal of $\Hom([d],[n])$.
Then $L(\cJ)$ has linear resolution.
The quotient map
\[ \kr[x_{[d] \times [n]}]/ L(\cJ) \mto{\hat{p_2}} 
\kr[x_{[n]}]/L^{p_2}(\cJ)\]
is a quotient map by a regular sequence, and $L^{p_2}(\cJ)$ is the
strongly stable ideal in $\kr[x_1, \ldots, x_n]$ associated to $\cJ$. 
\end{corollary}

The ideals $L(\cJ)$ are extensively studied by Nagel and Reiner in \cite{NaRe}.
Poset ideals $\cJ$ of $\Hom([d],[n])$ are there called strongly
stable $d$-uniform hypergraphs, \cite[Def. 3.3]{NaRe}.
If $M$ is the hypergraph corresponding to $\cJ$, the ideal $L(\cJ)$
is the ideal $I(F(M))$ of the $d$-partite $d$-uniform hypergraph $F(M)$
of \cite[Def.3.4, Ex.3.5]{NaRe}.

Furthermore the ideal $L^{p_2}(\cJ)$ is the ideal $I(M)$ of 
\cite[Ex. 3.5]{NaRe}. The squarefree ideal $I(K)$ of \cite[Ex.3.5]{NaRe}
is the ideal $L^\phi(\cJ)$ obtained from the map:
\begin{align*}
\phi : [d] \times [n] & \pil [d+n-1] \\
(a,b) & \mapsto a+b-1
\end{align*}
Corollary \ref{OrdCorSs} above is a part of \cite[Thm. 3.13]{NaRe}.

\medskip
Given a sequence $0 = a_0 \leq a_1 \leq \cdots \leq a_{d-1}$, 
we get an isotone map
\begin{eqnarray*}
\alpha : [d] \times [n] && \pil [n + a_{d-1}] \\
(i,j) && \mapsto j + a_{i-1} 
\end{eqnarray*}
having left strict chain fibers. The ideal
$L^\alpha(\cJ)$ is the ideal coming from the strongly stable ideal
associated to $\cJ$ by the stable operator
of S.Murai \cite[p.707]{MuSqu}. When $a_{i-1} < a_i$
they are called alternative squarefree operators in \cite[Sec. 4]{YaAlt}.

\begin{remark}
In  \cite{FrMe} Francisco, Mermin and Schweig 
consider a poset $Q$ with underlying set $\{1,2, \ldots, n\}$ 
where $Q$ is a weakening of the natural total order, and study
$Q$-Borel ideals. This is not quite within our setting, but adds
extra structure: Isotone maps $\phi : [d] \pil [n]$ uses the total
order on $[n]$ but
when studying poset ideals $\cJ$ the weaker poset structure
$Q$ is used on the codomain.
\end{remark}

\medskip
Let $\dis{n}$ be the poset which is the  disjoint union of 
the one element posets $\{1\},\ldots \{n\}$, 
so any two distinct elements are incomparable. This is the antichain 
on $n$ elements.

\subsection{Ferrers ideals:}{\it Poset ideals in $\Hom(\dis{2}, [n])$.}
By (\ref{FacetLigPart}) partitions $\la_1 \geq \cdots \geq \la_n \geq 0$
where $\la_1 \leq n$ correspond to elements of:
\[ \Hom([n], [n+1]) \iso \Hom([n] \times [n], [2]). \]
Thus $\la$ gives a poset ideal $\cJ$ in $[n] \times [n] = \Hom(\dis{2}, [n])$. 
The Ferrers ideal $I_\la$ of \cite[Sec. 2]{CoNa1} 
is the ideal $L(\cJ)$ in $\kr[x_{\dis{2} \times [n]}]$. 
In particular we recover the result from \cite[Cor.3.8]{CoNa1} 
that it has linear resolution.

More generally, the poset ideals $\cJ$ of $\Hom(\dis{d},[n])$ 
correspond to the $d$-partite $d$-uniform Ferrers hypergraphs $F$ in
\cite[Def. 3.6]{NaRe}. That $L(\cJ)$ has linear resolution is 
\cite[Thm. 3.13]{NaRe}.

\subsection{Edge ideals of  cointerval $d$-hypergraphs} 
Let $\Hom_s(Q,P)$ be strict isotone maps $\phi$, i.e.
$q < q^{\prime}$ implies $\phi(q) < \phi(q^{\prime})$.
There is an isomorphism of posets
\begin{equation} \label{OrdLigStrict}
 \Hom([d], [n]) \iso \Hom_s([d], [n+d-1]),
\end{equation}
by sending $\phi$ to $\phi_s$ given by $\phi_s(j) = \phi(j) + j -1$. 

Consider the weaker partial order on $\ordl$ on $\Hom([d],[n])$ 
where $\phi \ordl \psi$ if $\phi(i) \leq \psi(i)$ for $i < d$
and $\phi(d) = \psi(d)$. 
Via the isomorphism (\ref{OrdLigStrict}) 
this gives a partial order $\ordls$ on $\Hom_s([d], [n+d-1])$. 
The poset ideals for the partial order $\ordls$ correspond
to the cointerval $d$-hypergraphs of \cite[Def. 4.1]{DoEn}
on the set $\{1, 2, \ldots, n+d-1 \}$. Let $\cJ_s$ be such 
a poset ideal for $\ordls$. It corresponds to a poset ideal $\cJ$
in $\Hom([d],[n])$ for $\ordl$. 
Let \begin{align}  \label{OrdLigdn}
\phi: [d] \times [n] & \pil [d+n-1] \\
 (a,b) & \mapsto a+b-1 \notag
\end{align}
The ideal  $L^\phi(\cJ)$ is the edge ideal of the cointerval hypergraph
corresponding to $\cJ_s$, see \cite[Def. 2.1]{DoEn}.
By remarks \ref{CLPRemWOrd} and \ref{RegRemWOrd},  theorems \ref{CLPThmPNRJ}
and \ref{CLPThmPosid} 
still holds for the weaker partial order $\ordl$. Hence 
we recover the fact from \cite[Cor. 4.7]{DoEn} that edge ideals of
cointerval hypergraphs have linear resolution.
In the case $d = 2$ these ideals are studied also in 
\cite[Sec. 4]{CoNa1} and \cite[Sec. 2]{NaRe}. These are obtained
by cutting down by a regular sequence of differences of variables  
from a {\it skew} Ferrers ideals $I_{\la - \mu}$. The skewness implies 
the ideal comes from a poset ideal of $\Hom([2], [n])$ rather
than $\Hom(\dis{2}, [n])$. Due to this we get the map (\ref{OrdLigdn})
which has left strict chain fibers, and so the ideal
$\overline{I_{\la - \mu}}$, of \cite[Sec. 4]{CoNa1}.

\subsection{Uniform face ideals:} \label{ExCLPSubsecFace}
{\it Poset ideals in $\Hom(\dis{n}, [2])$.}
The uniform face ideal of a simplicial complex $\Delta$, introduced
recently by D.Cook \cite{Co}, see also \cite{HeHi}, 
is the ideal generated by the monomials
\[ \underset{i \in F}{\prod} x_i \cdot \underset{i \not \in F}{\prod} y_i\]
as $F$ varies among the faces of $\Delta$. 
The Boolean poset on $n$ elements 
is the distributive lattice $D(\dis{n}) = \Hom(\dis{n},[2])$. 
A simplicial complex $\Delta$ on the set $\{1, 2, \ldots, n \}$ 
corresponds to a poset ideal $\cJ$ of $\Hom(\dis{n}, [2])$, and
the uniform face ideal of $\Delta$ identifies as the subideal $L(\cJ)$ 
of $L(\dis{n}, [2])$. 

More generally Cook considers a set of 
vertices which is a disjoint union of $k$ ordered sets 
$\cC_1 \cup \cdots \cup \cC_k$, each $\cC_i$ considered a colour class.
He then considers simplicial complexes $\Delta$ which are {\it nested}
with respect to these orders \cite[Def.3.3, Prop.3.4]{Co}. He
associates to this the {\it uniform face ideal} $I(\Delta, \cC)$,
\cite[Def. 4.2]{Co}. Let $c_i$ be the cardinality of $\cC_i$ and consider
the poset which is the disjoint union $\cup_{i = 1}^k [c_i]$. 
Then such a $\Delta$ corresponds precisely to a poset ideal $\cJ$ in
$\Hom(\cup_1^k [c_i], [2])$. In fact $\cJ$ is isomorphic to the
index poset $P(\Delta, \cC)$ of \cite[Def. 6.1]{Co}. 
The uniform face ideal is obtained as follows:
There are projection maps $p_i : [c_i] \times [2] \pil [2]$ and
so
\[ \cup_1^k p_i :  ( \cup_1^k [c_i]) \times [2] \pil \cup_1^k [2]. \]
This map has left strict chain fibers
and the ideal $L^{\cup_1^k p_i}( \cup_1^k [c_i], [2])$ is 
exactly the uniform face ideal $I(\Delta, \cC)$.
In \cite[Thm. 6.8]{Co} it is stated
that this ideal has linear resolution.

\medskip Returning again to the first case of the ideal $L(\cJ)$ in 
$L(\dis{n},[2])$, its Alexander dual is by Theorem 
\ref{CLPThmAd}:
\[L(\cJ)^A = L([2], \dis{n}) + K(\cJ^c). \]
Here $L([2],\dis{n})$ is the complete intersection of
$x_{1j}x_{2j}$ for $j = 1, \ldots, n$, while $K(\cJ^c)$ is 
generated by $\prod_{j \in G} x_{1j}$ where $G$ is a nonface of $\Delta$. 
Thus $K(\cJ^c)$ is the associated ideal
$I_\Delta \sus \kr[x_{11}, \ldots, x_{1n}]$. This is \cite[Thm.1.1]{HeHi}:
$L(\cJ)^A$ is the Stanley-Reisner ideal $I_\Delta$ with whiskers
$x_{1j}x_{2j}$. 

\section{Proof concerning Alexander duality}
\label{AdSec}

In this section  we prove Theorem \ref{QLPThmAd} concerning
the compatibility between Alexander duality and cutting
down by a regular sequence.
The following lemma holds for squarefree ideals. Surprisingly
it does not hold for monomial ideals in general, for
instance for $(x_0^n, x_1^n)\sus k[x_0,x_1]$. 

\begin{lemma} \label{PrAdLemNulldiv}
Let $I \sus S$ be a squarefree monomial ideal and let $f \in S$
be such that $x_1 f = x_0 f$ considered in $S/I$. 
Then for every monomial $m$ in $f$ we have 
$x_1m = 0 = x_0m$ in $S/I$.
\end{lemma}

\begin{proof}
Write $f = x_0^a f_a + \cdots x_0 f_1 + f_0$
where each $f_i$ does not contain $x_0$. 
The terms in $(x_1- x_0) f = 0$ of degree $(a+1)$ in $x_0$,
are in $x_0^{a+1} f_a$, and so this is zero. Since $S/I$ is 
squarefree, $x_0 f_a$ is zero, and so $f = x_0^{a-1}f_{a-1} + \cdots $. 
We may continue and get $f = f_0$. But then again in
$(x_1 - x_0) f = 0$ the terms with $x_0$ degree $1$ is $x_0f_0$
and so this is zero. The upshot is that 
$x_0f = 0 = x_1f$. But then each of the multigraded terms
of these must be zero, and this gives the conclusion.
\end{proof}

Let $S$ be the polynomial ring $\kr[x_0, x_1, x_2, \ldots, x_n]$ and 
$I \sus S$
a squarefree monomial ideal with Alexander dual $J \sus S$. 
Let $S_1 = k[x,x_2, \ldots, x_n]$ and $S \pil S_1$ be the map
given by $x_i \mapsto x_i$ for $i \geq 2$ and $x_0,x_1 \mapsto x$. 

Let $I_1$ be the ideal of $S_1$ which is the image of $I$,
so the quotient ring of $S/I$ by the element $x_1 - x_0$ is the
ring $S_1 /I_1$. Similarly we define $J_1$. 
We now have the following. Part c. below is Theorem 3.1 in
the unpublished paper \cite{Loh}. 

\begin{proposition} \label{QLPProAlexDual}
a) If $x_1 - x_0$ is a regular element
of $S/I$, then $J_1$ is squarefree.

b) If $I_1$ is squarefree then $x_1 - x_0$ is a regular element
on $S/J$. 

c) If both $x_1 -x_0$ is a regular element on $S/I$ and $I_1$
is squarefree, then $J_1$ is the Alexander dual of $I_1$. 
\end{proposition}

\begin{proof}
The Alexander dual $J$ of $I$ consists of all monomials
in $S$ with non-trivial common factor (ntcf.) with all monomials in $I$. 

a) Let $F$ be a facet of the simplicial complex of $I$.
Let $m_F = \prod_{i \in F} x_i$. Suppose $F$ does not contain
any of the vertices $0$ and $1$. Then $x_1m_F = 0 = x_0 m_F$ in $S/I$ (since $F$
is a facet). Since $x_1 - x_0$ is regular we get $m_F = 0$ in $S/I$,
a contradiction. Thus every facet $F$ contains either $0$ or $1$. 
The generators of $J$ are $\prod_{i \in [n] \backslash F} x_i$,
and so no such monomial contains $x_0x_1$ and therefore
$J_1$ will be squarefree.

b) Suppose $(x_1 - x_0)f = 0$ in $S/J$. By the above
for the monomials $m$ in $f$, we have 
$x_1 m = 0 = x_0m$ in $S/J$. We may assume $m$ is squarefree.
So $x_0m$ has ntcf. with all monomials in $I$ and the same goes
for $x_1m$. 
If $m$ does not have ntcf. with the minimal monomial generator
$n$ in $I$, we must then have $n = x_0x_1n^\prime$. 
But then the image of $n$ in 
$I_1$ would not be squarefree, contrary to the assumption.
The upshot is that $m$ has ntcf. with every monomial in $I$
and so is zero in $S/J$. 

c) A monomial $m$ in $J$ has ntcf. with all monomials in $I$.
Then its image $\ovm$ in $S_1$ has ntcf. with all monomials in $I_1$, and
so $J_1$ is contained in the Alexander dual of $I_1$. 

Assume now $\ovm$ in $S_1$ has ntcf. with all monomials in $I_1$. 
We must show that $\ovm \in J_1$.
If $\ovm$ does not contain $x$ then $m$ has ntcf. with every monomial
in $I$, and so $\ovm \in J_1$. 

Otherwise $\ovm = x \ovmp$ and so $\ovm \in J_1$. We will show that either
$x_0m^\prime$ or $x_1 m^\prime$ is in $J$. 
If not, then $x_0m^\prime$ has no common factor with some monomial
$x_1n_1$ in $I$
(it must contain $x_1$ since $\ovm$ has ntcf. with every monomial in 
$I_1$), and $x_1m^\prime$ has no common factor with
some monomial $x_0 n_0$ in $I$. Let $n$ be the least common
multiple of $n_0$ and $n_1$. Then $x_0n$ and $x_1n$
are both in $I$ and so by the regularity assumption $n \in I$. 
But $n$ has no common factor with $x_0m^\prime$ and $x_1 m^\prime$,
and so $\ov{n} \in I_1$ has no common factor with $\ovm = x \ovmp$.
This is a contradiction. Hence either $x_0m^\prime$ or $x_1m^\prime$ is in
$J$ and so $\ovm$ is in $J_1$. 
\end{proof}


We are ready to round off this section by the following extension of
Theorem \ref{QLPThmAd}.

\begin{theorem} \label{QLPThmAdJ}
Let $\phi : [n] \times P \pil R$ be an isotone map such that
the fibers $\phi^{-1}(r)$ are bistrict chains in $[n] \times P^\op$.
Then the ideals $L^\phi(n,P;\cJ)$ and $L^{\phi^\tau}(P,n;\cJ)$ are
Alexander dual.
\end{theorem}

\begin{proof}
Since $L(P,n)$ is squarefree, the subideal $L(P,n;\cJ)$ is also.
Furthermore it is obtained by cutting down by a regular sequence of 
variable differences, by Theorem \ref{CLPThmNPRJ}.
Hence $L(n,P;\cJ)^\phi$ is the Alexander dual of $L(P,n;\cJ)^{\phi^\tau}$.
\end{proof}

\section{Proof that the poset maps induce regular sequences.}
\label{RegSec}
To prove Theorems \ref{QLPThmNPR}, \ref{QLPThmPNR} and
\ref{CLPThmPNRJ},
we will use an induction argument.
Let $[n] \times P \mto{\phi} R$ be an isotone map. 
Let $r \in R$ have inverse image by $\phi$ of cardinality $\geq 2$. 
Choose a partition into nonempty subsets 
$\phi^{-1}(r) = R_1  \cup R_2$ such that  $(i,p) \in R_1$ 
and $(j,q) \in R_2$ implies $i < j$. 
Let $R^\prime$ be $R \backslash \{ r \} \cup \{ r_1, r_2 \}$. 
We get the map 
\begin{equation} 
\label{QLPLigFactphi} [n] \times P \mto{\phi^\prime} R^\prime \pil R
\end{equation}
factoring $\phi$, where the elements of $R_i$ map to $r_i$. 
For an element $p^\prime$ of $R^\prime$, denote by $\ov{p^\prime}$
its image in $R$. Let $p^\prime, q^\prime$ be distinct elements of $R^\prime$. 
We define a partial order on $R^\prime$  by the following
two types of strict inequalities:

\begin{itemize}
\item[a.] $p^\prime < q^\prime$ if $p^\prime = r_1$ and $q^\prime = r_2$, 
\item[b.] $p^\prime < q^\prime$ if $\ov{p^\prime} < \ov{q^\prime}$
\end{itemize}

\begin{lemma}
This ordering is a partial order on $R^\prime$. 
\end{lemma}

\begin{proof}
Transitivity:
Suppose $p^\prime \leq q^\prime$ and $q^\prime \leq r^\prime$.
Then $\ov{p^\prime} \leq \ov{q^\prime}$ and $\ov{q^\prime} \leq \ov{r^\prime}$
and so $\ov{p^\prime} \leq \ov{r^\prime}$. If either $\ov{p^\prime}$ or
$\ov{r^\prime}$  is distinct
from $r$ we conclude that $p^\prime \leq r^\prime$. 
If both of them are equal to $r$, then $\ov{q^\prime} = r$ also. Then  either
$p^\prime = q^\prime = r^\prime$ or $p^\prime = r_1$ and $r^\prime = r_2$, 
and so $p^\prime \leq r^\prime$.

\medskip
Reflexivity: Suppose $p^\prime \leq q^\prime$ and $q^\prime \leq p^\prime$. 
Then $\ov{p^\prime} = \ov{q^\prime}$. If this is not $r$ we get $p^\prime
= q^\prime$. If it equals $r$, then since we do not have $r_2 \leq r_1$,
we must have again have $p^\prime = q^\prime$.
\end{proof}




\begin{proof}[Proof of Theorem \ref{QLPThmNPR}]
We show this by induction on the cardinality of $\im \phi$. Assume that 
we have a factorization (\ref{QLPLigFactphi}), such that
\begin{equation} \label{QLPLigKxR}
\kr[x_{R^\prime}]/ L^{\phi^\prime}(n,P)
\end{equation}
is obtained by cutting down
from $\kr[x_{[n] \times P}]/L(n,P)$ by a regular sequence
of variable differences. 

For $(a,p)$ in $[n] \times P$ denote its image in $R^\prime$ by $\ov{a,p}$ and
its image in $R$ by $\ovv{a,p}$. 
Let $(a,p)$ map to $r_1 \in R^\prime$ and $(b,q)$ map to $r_2 \in R^\prime$. 
We will show that $x_{r_1} - x_{r_2}$ is a regular element in
the quotient ring 
(\ref{QLPLigKxR}). 
So let $f$ be a polynomial of this quotient ring such that
$f(x_{r_1} - x_{r_2}) = 0$. 
Then by Lemma \ref{PrAdLemNulldiv}, for any monomial $m$ in $f$
we have $m x_{r_1} = 0 = m x_{r_2}$ in the quotient ring (\ref{QLPLigKxR}). 
We may assume $m$ is nonzero in this quotient ring.

There is a monomial 
$x_{\ov{1,p_1}} x_{\ov{1,p_2}} \cdots x_{\ov{n,p_n}}$ in $L^{\phi^\prime}(n,P)$
dividing $m x_{\ov{a,p}}$ considered as monomials in $\kr[x_R]$. 
Then we must have $\ov{a,p} = \ov{s,p_s}$ for some $s$. 
Furthermore there is $x_{\ov{1,q_1}} x_{\ov{1,q_2}} \cdots x_{\ov{n,q_n}}$
in $L^{\phi^\prime}(n,P)$
dividing $m x_{\ov{b,q}}$ in $\kr[x_R]$, and so
$\ov{b,q} = \ov{t, q_t}$ for some $t$. 

In $R$ we now get 
\[ \ovv{s,p_s} = \ovv{a,p} = \ovv{b,q} = \ovv{t,q_t}, \]
so $s = t$ would imply $q_t = p_s$ since $\phi$ has left strict chain fibers. 
But then
\[ r_1 = \ov{a,p} = \ov{s,p_s} = \ov{t,q_t} = \ov{b,q} = r_2 \]
which is not so.
Assume, say $s < t$. Then $p_s \geq q_t$ 
since $\phi$ has left strict chain fibers,
and so 
\[ p_t \geq p_s \geq q_t \geq q_s. \] 

\medskip
1. Suppose $p_s > q_t$. Consider 
$x_{\ov{1,q}} \cdots x_{\ov{t-1,q_{t-1}}}$. This will divide $m$
since $x_{\ov{1,q_1}} x_{\ov{1,q_2}} \cdots x_{\ov{n,q_n}}$ divides
$m x_{\ov{t,q_t}}$.
Similarly $x_{\ov{s+1,p_{s+1}}} \cdots x_{\ov{n,p_n}}$ divides $m$. 
Chose $s \leq r \leq t$. Then $p_r \geq p_s > q_t \geq q_r$
and so $\ov{r,q_r} < \ov{r,p_r}$. 
Then $x_{\ov{1,q}} \cdots x_{\ov{r-1,q_{r-1}}}$ and
$x_{\ov{r,p_{r}}} \cdots x_{\ov{n,p_n}}$ do not have a common
factor since 
\[ \ov{i,q_i} \leq \ov{r,q_r}  < \ov{r,p_r} \leq \ov{j,p_j} \]
for $i \leq r \leq j$.  Hence the product of these monomials will divide
$m$ and so $m = 0$ in the quotient ring $\kr[x_R] / L^\phi(n,P)$. 

\medskip
2. Assume $p_s = q_t$ and $q_t > q_s$. Then $\ov{s,p_s} > \ov{s,q_s}$
since $\phi$ has left strict chain fibers. 
The monomials 
$x_{\ov{1,q_{1}}} \cdots x_{\ov{s,q_s}}$ and 
$x_{\ov{s+1,p_{s+1}}} \cdots x_{\ov{n,p_n}}$ then do not have
any common factor, and the product divides $m$, showing that
$m = 0$ in the quotient ring $\kr[x_R] / L^\phi(n,P)$.

If $p_t > p_s$ we may argue similarly.

\medskip
3. Assume now that $p_t = p_s = q_t = q_s$, and denote this element as $p$.
Note that for $s \leq i \leq t$ we then have $p_s \leq p_i \leq p_t$,
so $p_i = p$, and the same argument shows that $q_i = p$ for $i$
in this range.

Since 
\[ \ov{s,p} = \ov{s,p_s} = \ov{a,p} \neq \ov{b,q} = \ov{t,q_t} = \ov{t,p}\]
there is $s \leq r \leq t$ such that 
\[ \ov{s,p_s} = \cdots = \ov{r,p_r} < \ov{r+1, p_{r+1}} 
\leq \cdots \leq \ov{t,p_t}.
\]
This is the same sequence as
\[ \ov{s,q_s} = \cdots = \ov{r,q_r} < \ov{r+1, q_{r+1}} \leq \cdots 
\leq \ov{t,q_t}. \]
Then $x_{\ov{1,q_{1}}} \cdots x_{\ov{r,q_r}}$ and 
$x_{\ov{r+1,p_{r+1}}} \cdots x_{\ov{n,p_n}}$ divide $m$ and do not have
a common factor, and so $m = 0$ in the quotient ring $\kr[x_R] / L^\phi(n,P)$.
\end{proof}

\begin{proof}[Proof of Theorem \ref{CLPThmNPRJ}]
There is a surjection
\begin{equation} \label{ProofNPRLig}
\kr[x_{[n] \times P}]/L(n,P) \pil \kr[x_{[n] \times P}]/L(\cJ)^A 
\end{equation}
and both are Cohen-Macaulay quotient rings of $\kr[x_{[n] \times P}]$
of codimension $|P|$. The basis $B$ is a regular sequence for the
first ring by the previous argument. Hence
\[ \kr[x_{[n] \times P}]/(L(n,P) + (B))\]
has codimension $|P| + |B|$. Since 
\[ \kr[x_{[n] \times P}]/(L(\cJ)^A + (B))\]
is a quotient of this it must have codimension
$\geq |P| + |B|$. But then $B$ must be a regular sequence for
the right side of (\ref{ProofNPRLig}) also.
\end{proof}

\begin{proof}[Proof of Theorems \ref{QLPThmPNR} and \ref{CLPThmPNRJ}]
By induction on the cardinality of $\im \phi$. We assume we have a
factorization 
\begin{equation*} 
P \times [n] \mto{\phi^\prime} R^\prime \pil R
\end{equation*}
analogous to 
(\ref{QLPLigFactphi}), such that
\begin{equation} \label{QLPLigKxR2}
\kr[x_{R^\prime}]/ L^{\phi^\prime}(\cJ)
\end{equation}
is obtained by cutting down
from $\kr[x_{P \times [n]}]/L(\cJ)$ by a regular sequence
of variable differences. 

Let $(p_0,a)$ map to $r_1 \in R^\prime$ and $(q_0,b)$ map to $r_2 \in R^\prime$. 
We will show that $x_{r_1} - x_{r_2}$ is a regular element in
the quotient ring 
(\ref{QLPLigKxR2}). 
So let $f$ be a polynomial of this quotient ring such that
$f(x_{r_1} - x_{r_2}) = 0$. 
Then by Lemma \ref{PrAdLemNulldiv}, for any monomial $m$ in $f$
we have $m x_{r_1} = 0 = m x_{r_2}$ in the quotient ring 
$\kr[x_{R^\prime}]/ L^{\phi^\prime}(\cJ)$.
We may assume $m$ is nonzero in this quotient ring.

There is $i \in \cJ \sus \Hom(P,n)$ such that the
monomial $m^i = \prod_{p \in P} x_{\ov{p,i_p}}$ in 
$L^{\phi^\prime}(\cJ)$ divides $m x_{\ov{p_0,a}}$,
and similarly a $j \in \cJ$ such that the monomial
$m^j = \prod_{p \in P} x_{\ov{p,j_p}}$ divides $m x_{\ov{q_0,b}}$.
Hence there are $s$ and $t$ in $P$ such that
$\ov{s,i_s} = \ov{p_0,a}$ and $\ov{t,j_t} = \ov{q_0,b}$. 
In $R$ we then get:
\[ \ovv{s,i_s} = \ovv{p_0,a} = \ovv{q_0,b} = \ovv{t,j_t}, \]
so $s = t$ would imply $i_t = j_t$ since $\phi$ has left strict chain fibers.
But then
\[ r_1 = \ov{p_0,a} = \ov{s,i_s} = \ov{t,j_t} = \ov{q_0,b} = r_2 \]
which is not so.
Assume then, say $s < t$.
Then $i_s \geq j_t$ since $\phi$ has left strict chain fibers, and so
\begin{equation}
\label{ProofLigst}
i_t \geq i_s \geq j_t \geq j_s.
\end{equation}

Now form the monomials
\begin{itemize}
\item $m^i_{> s} = \prod_{p > s} x_{\ov{p,i_p}}$.
\item $m^i_{i > j} = \underset{i_p > j_p, not \, (p > s)}{\prod} x_{\ov{p, i_p}}$.
\item $m^i_{i < j} = \underset{i_p < j_p, not \, (p > s)}{\prod} x_{\ov{p, i_p}}$.
\item $m^i_{i = j} = \underset{i_p = j_p, not \, (p > s)}{\prod} x_{\ov{p, i_p}}$.
\end{itemize}
Similarly we define $m^j_*$ for the various subscripts $*$. 
Then 
\[ m^i = m^i_{i=j} \cdot m^i_{i > j} \cdot m^i_{i < j} \cdot m^i_{>s} \]
divides $x_{\ov{s, i_s}} m$, and
\[ m^j = m^j_{i=j} \cdot m^j_{i > j} \cdot m^j_{i < j} \cdot m^j_{>s} \]
divides $x_{\ov{t, j_t}} m$.

There is now a map $\ell : P \pil [n]$ defined by
\[ \ell(p) = \begin{cases} i_p \, \mbox{ for } p > s \\
                       \min(i_p, j_p) \, \mbox{ for not } (p > s)
          \end{cases}.
\]
This is an isotone map as is easily checked. Its associated monomial
is 
\[ m^\ell = m_{i=j} \cdot m^j_{i > j} \cdot m^i_{i < j} \cdot m^i_{>s}. \]
We will show that this divides $m$. Since the isotone map $\ell$ is
$\leq$ the isotone map $i$, this will prove the theorem.

\medskip
\begin{claim} $m^j_{i> j}$ is relatively prime to 
$m^i_{i<j}$ and $m^i_{>s}$. 
\end{claim}

\begin{proof}
Let $x_{\ov{p, j_p}}$ be in $m^j_{i > j}$. 

1. Suppose it equals the variable
$x_{\ov{q, i_q}}$ in $m^i_{ i < j}$. Then $p$ and $q$ are comparable
since $\phi$ has left strict chain fibers.
If $p < q$ then $j_p \geq i_q \geq i_p$, contradicting $i_p > j_p$. 
If $q < p$ then $i_q \geq j_p \geq j_q$ contradicting $i_q < j_q$.

2 Suppose $x_{\ov{p, j_p}}$ equals $x_{\ov{q, i_q}}$ in $m^i_{> s}$. 
Then $p$ and $q$ are comparable and so $p < q$ since $q > s$ and we do not
have $p > s$. Then $j_p \geq i_q \geq i_p$ contradicting $i_p > j_p$. 
\end{proof}

Let $abc =  m^i_{i=j} \cdot m^i_{i < j} \cdot m^i_{>s}$ which divides
$m x_{\ov{s,i_s}}$ and $a b^\prime = m^j_{i=j} \cdot m^j_{i > j}$
which divides $m$ since $x_{\ov{t,j_t}}$ is a factor of $m^j_{>s}$
since $t > s$. 
Now if the product of monomials $abc$ divides the monomial $n$ 
and $ab^\prime$ also divides $n$, and $b^\prime$ is relatively prime to $bc$,
then the least common multiple $abb^\prime c$ divides $n$.
We thus see that the monomial associated to the isotone map $\ell$
\[ m^\ell = m_{i=j} \cdot m^j_{i > j} \cdot m^i_{i < j} \cdot m^i_{>s} \]
divides $m x_{\ov{s,i_s}}$. 
We need now only show that the variable $x_{\ov{s,i_s}}$ occurs to
a power in the above  product for $m^\ell$ less than or equal to that of its
power in $m$. 

\begin{claim} $x_{\ov{s,i_s}}$ is not a factor of
$m^j_{i > j}$ or $m^i_{i < j}$. 
\end{claim}

\begin{proof}
1. Suppose $\ov{s, i_s} = \ov{p, i_p}$ where $i_p < j_p$ 
and not $p > s$. Since $p$ and $s$ are comparable 
(they are both in a fiber of $\phi$), we have $p \leq s$.
Since $\phi$ is isotone $i_p \leq i_s$ and since $\phi$ has
left strict chain fibers $i_p \geq i_s$. Hence $i_p = i_s$.
By (\ref{ProofLigst}) $j_s \leq i_s$ and so 
$j_p \leq j_s \leq i_s = i_p$. This contradicts $i_p < j_p$. 

2. Suppose $\ov{s,i_s} = \ov{p, j_p}$ where $j_p < i_p$ and not $p > s$.
Then again $p \leq s$ and $i_p \leq i_s \leq j_p$, giving a contradiction.
\end{proof}

If now $i_s > j_s$ then $x_{\ov{s, i_s}}$ is a factor in
$m^i_{i>j}$ but by the above, not in $m^j_{i>j}$. Since $m^\ell$ is obtained from 
$m^i$ by replacing
$m^i_{i > j}$ with $m^j_{i > j}$, we see that $m^\ell$ contains
a lower power of $x_{\ov{s,i_s}}$ than $m^i$ and so $m^\ell$ divides $m$.


\begin{claim} \label{ProofClaimLik} Suppose $i_s = j_s$. Then the power of
$x_{\ov{s,i_s}}$ in $m^i_{ >s}$ is less than or equal to its power
in $m^j_{> s}$. 
\end{claim}

\begin{proof}
Suppose $\ov{s, i_s} = \ov{p, i_p}$ where $p > s$. 
We will show that then $i_p = j_p$. This will prove the claim.

The above implies $\ovv{p, i_p} = \ovv{s, i_s} = \ovv{t, j_t}$,
so either $s < p < t$ or $s < t \leq p$.
If the latter holds, then since $\phi$ has left strict chain fibers,
$i_s \geq j_t \geq i_p$ and also $i_s \leq i_p$ by isotonicity, 
and so $i_s = i_p = j_t$. Thus
\[ \ov{s,i_s} \leq \ov{t, j_t} \leq \ov{p, i_p} \]
and since the extremes are equal, all three are equal contradicting
the assumption that the two first are unequal.

Hence $s < p < t$. By assumption on the fibre of $\phi$ we have 
$i_s \geq i_p$ and by isotonicity $i_s \leq i_p$ and so $i_s = i_p$.
Also by (\ref{ProofLigst}) and isotonicity
\[ i_s \geq j_t \geq j_p \geq j_s .\]
Since $i_s = j_s$ we get equalities everywhere and so $i_p = j_p$,
as we wanted to prove. 
\end{proof}

In case $i_s > j_s$ we have shown before Claim \ref{ProofClaimLik} 
that $m^\ell$ divides $m$.
So suppose $i_s = j_s$. By the above two claims, the $x_{\ov{s,i_s}}$ 
in $m^\ell$ occurs only in $m_{i=j} \cdot m^i_{ > s}$ and to 
a power less than or equal to that in $m_{i = j} \cdot m^j_{>s}$. 
Since $m^j$ divides $m x_{\ov{t,j_t}}$ and 
$\ov{s,i_s} \neq \ov{t, j_t}$ the power of $\ov{s,i_s}$
in $m^j$ is less than or equal to its power in $m$. 
Hence the power of $x_{\ov{s,i_s}}$ in $m^\ell$ is less or
equal to its power in $m$ and $m^\ell$ divides $m$.

\end{proof}

\begin{remark} \label{RegRemWOrd} Suppose $P$ has a unique
maximal element $p$. 
The above proof still holds if $\cJ$ in $\Hom(P,[n])$ is a poset ideal
for the weaker partial order $\leq^w$ on $\Hom(P,[n])$ where
the isotone maps $\phi \leq^w \psi$ if $\phi(q) \leq \psi(q)$
for $q < p$, and $\phi(p) = \psi(p)$.
\end{remark}

\bibliographystyle{amsplain}
\bibliography{Bibliography}

\providecommand{\bysame}{\leavevmode\hbox to3em{\hrulefill}\thinspace}
\providecommand{\MR}{\relax\ifhmode\unskip\space\fi MR }
\providecommand{\MRhref}[2]{%
  \href{http://www.ams.org/mathscinet-getitem?mr=#1}{#2}
}
\providecommand{\href}[2]{#2}
\begin{thebibliography}{10}

\bibitem{AlBiHeLu}
Klaus Altmann, Mina Bigdeli, Juergen Herzog, and Dancheng Lu,
  \emph{Algebraically rigid simplicial complexes and graphs}, Journal of Pure
  and Applied Algebra \textbf{220} (2016), no.~8, 2914--2935.

\bibitem{BaSt}
David Bayer and Michael Stillman, \emph{{A criterion for detecting
  m-regularity}}, Inventiones mathematicae \textbf{87} (1987), no.~1, 1--11.

\bibitem{BPSZ}
Anders Bj{\"o}rner, Andreas Paffenholz, Jonas Sj{\"o}strand, and G{\"u}nter~M
  Ziegler, \emph{Bier spheres and posets}, Discrete \& Computational Geometry
  \textbf{34} (2005), no.~1, 71--86.

\bibitem{Bu}
David~A Buchsbaum, \emph{{Selections from the Letter-Place Panoply}},
  Commutative Algebra, Springer, 2013, pp.~237--284.

\bibitem{CHT}
Aldo Conca, Serkan Ho{\c{s}}ten, and Rekha~R. Thomas, \emph{Nice initial
  complexes of some classical ideals}, Algebraic and geometric combinatorics,
  Contemp. Math., vol. 423, Amer. Math. Soc., Providence, RI, 2006, pp.~11--42.

\bibitem{Co}
David Cook, \emph{The uniform face ideals of a simplicial complex}, arXiv
  preprint arXiv:1308.1299 (2013).

\bibitem{CoNa1}
Alberto Corso and Uwe Nagel, \emph{Specializations of {F}errers ideals}, J.
  Algebraic Combin. \textbf{28} (2008), no.~3, 425--437.

\bibitem{DaFlNeCoLP}
Alessio D'Ali, Gunnar Fl{\o}ystad, and Amin Nematbakhsh, \emph{Resolutions of
  co-letterplace ideals and generalizations of bier spheres}, arXiv preprint
  arXiv:1601.02793 (2016).

\bibitem{DaFlNeLP}
\bysame, \emph{Resolutions of letterplace ideals of posets}, arXiv preprint
  arXiv:1601.02792 (2016).

\bibitem{DoEn}
Anton Dochtermann and Alexander Engstr{\"o}m, \emph{Cellular resolutions of
  cointerval ideals}, Math. Z. \textbf{270} (2012), no.~1-2, 145--163.

\bibitem{ER}
John~A Eagon and Victor Reiner, \emph{{Resolutions of Stanley-Reisner rings and
  Alexander duality}}, Journal of Pure and Applied Algebra \textbf{130} (1998),
  no.~3, 265--275.

\bibitem{EaRe}
\bysame, \emph{{Resolutions of Stanley-Reisner rings and Alexander duality}},
  Journal of Pure and Applied Algebra \textbf{130} (1998), no.~3, 265--275.

\bibitem{HeEn}
Viviana Ene and J{\"u}rgen Herzog, \emph{Gr{\"o}bner bases in commutative
  algebra}, vol. 130, American Mathematical Soc., 2012.

\bibitem{EHM}
Viviana Ene, J{\"u}rgen Herzog, and Fatemeh Mohammadi, \emph{{Monomial ideals
  and toric rings of Hibi type arising from a finite poset}}, European Journal
  of Combinatorics \textbf{32} (2011), no.~3, 404--421.

\bibitem{Fl}
Gunnar Fl{\o}ystad, \emph{{Cellular resolutions of Cohen-Macaulay monomial
  ideals}}, J. Commut. Algebra \textbf{1} (2009), no.~1, 57--89.

\bibitem{FlNe}
Gunnar Fl{\o}ystad and Amin Nematbakhsh, \emph{Rigid ideals by deforming
  quadratic letterplace ideals}, arXiv preprint arXiv:1605.07417 (2016).

\bibitem{FlVa}
Gunnar Fl{\o}ystad and Jon~Eivind Vatne, \emph{{(Bi-)Cohen-Macaulay Simplicial
  Complexes and Their Associated Coherent Sheaves}}, Communications in algebra
  \textbf{33} (2005), no.~9, 3121--3136.

\bibitem{FrMe}
Christopher~A Francisco, Jeffrey Mermin, and Jay Schweig, \emph{{Generalizing
  the Borel property}}, Journal of the London Mathematical Society \textbf{87}
  (2013), no.~3, 724--740.

\bibitem{Green}
Mark~L Green, \emph{Generic initial ideals}, Six lectures on commutative
  algebra, Springer, 1998, pp.~119--186.

\bibitem{HeHi}
J.~{Herzog} and T.~{Hibi}, \emph{{The face ideal of a simplicial complex}},
  ArXiv e-prints.

\bibitem{HeHi2}
J{\"u}rgen Herzog and Takayuki Hibi, \emph{Distributive lattices, bipartite
  graphs and alexander duality}, Journal of Algebraic Combinatorics \textbf{22}
  (2005), no.~3, 289--302.

\bibitem{HeHiMon}
\bysame, \emph{Monomial ideals}, Springer, 2011.

\bibitem{HeQuSh}
J{\"u}rgen Herzog, Ayesha~Asloob Qureshi, and Akihiro Shikama, \emph{Alexander
  duality for monomial ideals associated with isotone maps between posets},
  Journal of Algebra and Its Applications \textbf{15} (2016), no.~05, 1650089.

\bibitem{HeRa}
J{\"u}rgen Herzog and Ahad Rahimi, \emph{Bi-cohen-macaulay graphs}, arXiv
  preprint arXiv:1508.07119 (2015).

\bibitem{HeTa}
J{\"u}rgen Herzog and Yukihide Takayama, \emph{Resolutions by mapping cones},
  Homology, Homotopy and Applications \textbf{4} (2002), no.~2, 277--294.

\bibitem{HeTr}
J{\"u}rgen Herzog and Ng{\^o}Vi{\^e}t Trung, \emph{Gr{\"o}bner bases and
  multiplicity of determinantal and pfaffian ideals}, Advances in Mathematics
  \textbf{96} (1992), no.~1, 1--37.

\bibitem{KJM}
Martina Juhnke-Kubitzke, Lukas Katth{\"a}n, and Sara~Saeedi Madani,
  \emph{Algebraic properties of ideals of poset homomorphisms}, Journal of
  Algebraic Combinatorics (2015), 1--28.

\bibitem{Kl}
Jan~O Kleppe, \emph{Deformations of modules of maximal grade and the hilbert
  scheme at determinantal schemes}, Journal of Algebra \textbf{407} (2014),
  246--274.

\bibitem{ScaLev}
Roberto La~Scala and Viktor Levandovskyy, \emph{{Letterplace ideals and
  non-commutative Gr{\"o}bner bases}}, Journal of Symbolic Computation
  \textbf{44} (2009), no.~10, 1374--1393.

\bibitem{Loh}
Henning Lohne, \emph{The many polarizations of powers of maximal ideals}, arXiv
  preprint arXiv:1303.5780 (2013).

\bibitem{MacL}
Saunders Mac~Lane, \emph{Categories for the working mathematician}, vol.~5,
  springer, 1998.

\bibitem{MuSqu}
Satoshi Murai, \emph{Generic initial ideals and squeezed spheres}, Adv. Math.
  \textbf{214} (2007), no.~2, 701--729.

\bibitem{NaRe}
Uwe Nagel and Victor Reiner, \emph{Betti numbers of monomial ideals and shifted
  skew shapes}, Electron. J. Combin. \textbf{16} (2009), no.~2, Special volume
  in honor of Anders Bjorner, Research Paper 3, 59.

\bibitem{Nar}
Himanee Narasimhan, \emph{The irreducibility of ladder determinantal
  varieties}, Journal of Algebra \textbf{102} (1986), no.~1, 162--185.

\bibitem{Sta}
Richard~P Stanley, \emph{{Algebraic Combinatorics: Walks, Trees, Tableaux and
  More}}, Undergraduate Texts in Mathematics (2013).

\bibitem{Stu}
Bernd Sturmfels, \emph{Gr{\"o}bner bases and stanley decompositions of
  determinantal rings}, Mathematische Zeitschrift \textbf{205} (1990), no.~1,
  137--144.

\bibitem{YaAlt}
Kohji Yanagawa, \emph{Alternative polarizations of {B}orel fixed ideals},
  Nagoya Math. J. \textbf{207} (2012), 79--93.

\end{thebibliography}

\end{document}